\documentclass[a4paper]{article}

\usepackage[left=2.7cm,right=2.7cm,top=3.2cm,bottom=3.2cm]{geometry}

\usepackage{amsmath,amsthm,amssymb}

\usepackage{dsfont}
\usepackage{extarrows}
\usepackage{enumitem}
\usepackage{authblk}
\usepackage{mathtools}
\usepackage{comment}

\usepackage[numbers,square,sort&compress]{natbib}

\usepackage[utf8]{inputenc}
\usepackage[colorlinks=true,urlcolor=blue,
citecolor=blue,linkcolor=blue,linktocpage,pdfpagelabels,
bookmarksnumbered,bookmarksopen]{hyperref}

\let\svthefootnote\thefootnote
\newcommand\freefootnote[1]{%
  \let\thefootnote\relax%
  \footnotetext{\noindent #1}%
  \let\thefootnote\svthefootnote%
}

\usepackage{graphicx}
\graphicspath{{fig/}}

\numberwithin{equation}{section}
\mathtoolsset{showonlyrefs}

\usepackage{algorithm, algpseudocode}

\usepackage{mathrsfs}

\usepackage{lipsum}
 \usepackage{xfrac}

\theoremstyle{plain}
\newtheorem{theorem}{Theorem}[section]
\newtheorem{lemma}[theorem]{Lemma}
\newtheorem{corollary}[theorem]{Corollary}
\newtheorem{proposition}[theorem]{Proposition}
\newtheorem{remark}{Remark}[section]

\usepackage{xcolor}

\title{Sharp Brezis--Merle Inequalities for Polyharmonic Dirichlet Problems: Extremal Measures and Critical Asymptotics}

\author[1]{Jo\~ao Marcos do \'O}
\author[2]{Federica Sani}
\author[3]{Victor Rafael Silva}

\affil[1]{\small{Federal University of Para\'{\i}ba, Brazil\\
  \texttt{jmbo@mat.ufpb.br}}}
\affil[2]{Universit\`{a} di Modena e Reggio Emilia, Italy\\
  \texttt{federica.sani@unimore.it}}
\affil[3]{Federal University of Para\'{\i}ba, Brazil\\
  \texttt{victor.rafael@academico.ufpb.br
}}

\date{}

\begin{document}

\maketitle


\begin{abstract}
We establish the Brezis--Merle inequality for very weak solutions of 
polyharmonic Dirichlet problems with integrable data on smooth bounded domains 
in conformal dimensions. The inequality is sharp both with respect to the 
admissible range of the exponent and to the quantity appearing on the 
right-hand side, which is characterized as the maximum of an associated 
one-pole Green functional. The supremum defining this quantity is not attained 
by any integrable datum, but remains unchanged when the admissible class is 
extended to Radon probability measures; in this enlarged class, the extremals 
are Dirac masses, and their concentration points are characterized
in terms of the associated Green function. On the unit ball, we prove uniqueness of the 
extremal measure and obtain an explicit optimal bound extending the sharp 
result for the Laplacian. Finally, we determine the critical asymptotics of 
the sharp constant on the right-hand side and show that the possible limiting 
concentration points are selected by the maximum points of the polyharmonic 
Robin function.
\end{abstract}
 \freefootnote{\textbf{Keywords:} Brezis--Merle inequality, polyharmonic operator, exponential integrability,
extremal measures, Green functions, Robin functions.}

\freefootnote{\textbf{2020 Mathematics Subject Classification: }35A23, 35B33, 35J40, 35J08, 31B30, 35J91.
}

    \tableofcontents


\section{Introduction}

The central aim of this work is to establish the optimal form of the
Brezis--Merle inequality for polyharmonic Dirichlet problems, with
$L^{1}$-data in conformal dimensions, of the form
\begin{equation}
\tag{$P_{\Omega}$}
\label{POMEGA}
\begin{cases}
(-\Delta)^m u=f(x)
    & \text{in }\Omega,\\
\partial_\nu^j u=0
    & \text{on }\partial\Omega,
      \qquad j=0,\ldots,m-1,
\end{cases}
\end{equation}
where $\Omega\subset\mathds{R}^{2m}$ $(m \geq 1)$ is a smooth bounded
domain, $f\in L^1(\Omega)$, and $\nu$ denotes the outer unit normal to $\partial\Omega$.

A function $u\in L^1(\Omega)$ is called a very weak solution of
\eqref{POMEGA} if
\begin{equation}
\label{very weak solution2}
\int_\Omega u(-\Delta)^m\varphi\,\mathrm{d}x
=
\int_\Omega f\varphi\,\mathrm{d}x
\qquad
\text{for every }
\varphi\in C_{\mathrm{cl}}^{2m}(\overline{\Omega}),
\end{equation}
where $C_{\mathrm{cl}}^{2m}(\overline{\Omega})$ is the Dirichlet test class associated with \eqref{POMEGA}, precisely,
\begin{equation}
\label{test space}
C_{\mathrm{cl}}^{2m}(\overline{\Omega})
:=
\left\{
\varphi\in C^{2m}(\overline{\Omega}) :
\partial_\nu^j\varphi=0
\text{ on }\partial\Omega,\quad
j=0,\ldots,m-1
\right\}.
\end{equation}

The main result is the following.
\begin{theorem}\label{main sharp BM}
Let $\Omega\subset\mathds{R}^{2m}$ ($m \geq 1$) be a smooth bounded
domain. If $f\in L^1(\Omega)\setminus\{0\}$ and $u$ is a  very weak solution of
\eqref{POMEGA} then
\begin{equation}\label{sharp BM}
\int_{\Omega}
\exp\left(
\frac{\alpha|u(x)|}{\lVert f\rVert_{L^1(\Omega)}}
\right)\,\mathrm{d}x
\leq
\max_{y\in\Omega}
\int_{\Omega}
e^{\alpha|G_{\Omega}(x,y)|}\,\mathrm{d}x,
\end{equation}
for all $\alpha\in(0,\Lambda_m)$, where
$$    \Lambda_m:=
    2^{2m}\pi^m m!
$$
and $G_{\Omega}$ is the Green function  of $(-\Delta)^m$ in $\Omega$ under
homogeneous Dirichlet boundary conditions.

The range of the exponent $\alpha\in(0,\Lambda_m)$ is sharp: if
$\alpha\geq\Lambda_m$, then, for every $C>0$, there exists
$f\in L^1(\Omega)\setminus\{0\}$ such that the corresponding very weak
solution $u$ satisfies
\[
\int_{\Omega}
\exp\left(
\frac{\alpha|u(x)|}{\lVert f\rVert_{L^1(\Omega)}}
\right)\,\mathrm{d}x
>
C.
\]
Moreover, for each fixed $\alpha\in(0,\Lambda_m)$, the constant on the right-hand side
of \eqref{sharp BM} is optimal: the constant $\max_{y\in\Omega}
\int_{\Omega}
e^{\alpha|G_{\Omega}(x,y)|}\,\mathrm{d}x$ is well defined and is the smallest constant, independent of $f$, for
which \eqref{sharp BM} holds for every nonzero datum
$f\in L^1(\Omega)$.
\end{theorem}

\begin{remark}
By saying that \emph{the constant $\max_{y\in\Omega}
\int_{\Omega}
e^{\alpha|G_{\Omega}(x,y)|}\,\mathrm{d}x$ is well defined}, we mean that the function $y \in \Omega \mapsto \int_{\Omega}
e^{\alpha|G_{\Omega}(x,y)|}\,\mathrm{d}x$ attains its maximum on $\Omega$. This will be proved in \hyperref[Jalpha sup attained]{Proposition \ref{Jalpha sup attained}}.
\end{remark}

The space of test functions  defined in \eqref{test space} is the natural choice to encode the boundary conditions of the Dirichlet polyharmonic problem \eqref{POMEGA}. Indeed, suppose that
\(u\) is smooth and satisfies \eqref{POMEGA}. Then, by the polyharmonic version of the integration-by-parts formula given in the classical work of Boggio \cite[p.~101, formula~(3)]{Boggioclassical}, the following identity holds, for every $\varphi \in C^{2m}(\overline{\Omega})$,
\begin{equation}
\label{eq:polyharmonic-green-identity}
\begin{aligned}
&\int_\Omega
\left[
u(-\Delta)^m\varphi
-
\varphi(-\Delta)^m u
\right]\,\mathrm{d}x
\\
&\qquad =
(-1)^m
\sum_{j=0}^{m-1}
\int_{\partial\Omega}
\left[
\Delta^j u\,
\partial_\nu\left(\Delta^{m-1-j}\varphi\right)
-
\partial_\nu\left(\Delta^j u\right)
\Delta^{m-1-j}\varphi
\right]\,\mathrm{d}S.
\end{aligned}
\end{equation}
It follows that if $\varphi \in C^{2m}(\overline{\Omega})$ satisfies the homogeneous
Dirichlet conditions of orders $0,\ldots,m-1$, namely if $\varphi \in C_{\mathrm{cl}}^{2m}(\overline{\Omega})$, the boundary integrals
in \eqref{eq:polyharmonic-green-identity} vanish. Consequently,
\[
\int_\Omega u(-\Delta)^m\varphi\,\mathrm{d}x
=
\int_\Omega \varphi f\,\mathrm{d}x \qquad \text{ for every }  \varphi \in C_{\mathrm{cl}}^{2m}(\overline{\Omega}).
\]
Thus, every classical solution of \eqref{POMEGA} satisfies
\eqref{very weak solution2}. Moreover, the latter identity
remains meaningful when $u,f\in L^{1}(\Omega)$ and
retains the boundary conditions of problem \eqref{POMEGA}.

\medskip

In dimension two, Brezis and Merle \cite{BM} established exponential
integrability for very weak solutions of the Dirichlet problem for the
Laplacian with $L^1$-data. Their estimate became a fundamental tool in
the study of Liouville-type equations, compactness, and blow-up
phenomena.

A functional-analytic formulation of this phenomenon was subsequently
developed by Cassani, Ruf, and Tarsi \cite{CassaniRufTarsi}. For the
$\Delta$-reduced Sobolev space in dimension two
\[
W_{\Delta}^{2,1}(\Omega)
:=
\left\{
u\in W_0^{1,1}(\Omega):
\Delta u\in L^1(\Omega)
\right\},
\]
they proved the sharp limiting embedding into the Zygmund space
$L_{\exp}(\Omega)$, which enabled them to deduce the validity of a Brezis--Merle-type inequality for more general integrals. Fontana and Morpurgo \cite{FontanaMorpurgo} sharpened  this analysis and, as a consequence, obtained the explicit
Brezis--Merle estimate
\begin{equation}\label{eq:fontana-morpurgo-bound}
\int_\Omega
\exp\left(
\frac{\alpha |u(x)|}{\|\Delta u\|_{L^1(\Omega)}}
\right)\,\mathrm{d}x
\leq
\frac{4\pi}{4\pi-\alpha}\,|\Omega|,
\qquad
0<\alpha<4\pi,
\end{equation}
for all $u \in W_{\Delta}^{2,1}(\Omega)$. A corresponding estimate, formulated for weak solutions of second-order elliptic equations with finite Radon measure data, had previously been obtained by Alberico and Ferone \cite[Theorem 3.1]{Alberico}. Fontana and Morpurgo later proved that the constant on the right-hand side of \eqref{eq:fontana-morpurgo-bound} cannot be decreased \emph{at least} when $\Omega$ is a ball (see \cite[Corollary 2]{FontanaMorpurgo}).

For higher-order operators, the conformal dimension is $2m$, where
the fundamental solution of $(-\Delta)^m$ in $\mathds{R}^{2m}$ has a
logarithmic singularity. In this setting, the choice of boundary conditions plays a crucial role. Under Navier conditions, $\Delta^j u=0$ on
$\partial\Omega$ for $j=0,\ldots,m-1$, the equation can be rewritten
as a system of second-order Dirichlet problems, so that positivity and
comparison arguments remain available. A Brezis--Merle-type estimate
for the biharmonic Navier problem in dimension four was established by
Lin \cite[Lemma~2.3]{Lin1998}, and its polyharmonic counterpart in any
conformal dimension  is due to Martinazzi \cite{MartinazziBMNavier}. No
analogous reduction is available under the homogeneous Dirichlet
conditions considered here, and the corresponding polyharmonic Green
function may change sign. In our previous work \cite{PreviousPaper}, we treated the biharmonic Dirichlet
problem in four-dimensional balls and, using Boggio's positive
Green function \cite{Boggioclassical}, proved the Brezis--Merle
inequality with sharp critical exponent $32\pi^2$, together with
applications to semilinear biharmonic equations and non-variational
cooperative systems.

The present paper develops this
Dirichlet theory for the  polyharmonic operator on general smooth
bounded domains. For arbitrary integer $m \geq 1$, we determine both the sharp
critical exponent and the best possible right-hand side, prove that
equality with this bound is not attained by any $L^1$-datum,
characterize all probability measures for which equality holds, and
analyze the behavior of the sharp constant on the right-hand side and of the extremal measures
as the exponent $\alpha$ approaches the critical threshold $\Lambda_m$. Our result is partially new even when $m=1$; indeed, we determine the best possible
right-hand side in the Brezis--Merle inequality on any
smooth bounded domain. In balls, it coincides with the sharp bound found by
Fontana and Morpurgo (see \eqref{eq:fontana-morpurgo-bound} and \cite[Corollary 2]{FontanaMorpurgo}). Our methods are different from those employed in
\cite{Alberico,CassaniRufTarsi,FontanaMorpurgo}: while these works rely
on symmetrization and decreasing-rearrangement techniques, with
\cite{CassaniRufTarsi,FontanaMorpurgo} being formulated in terms of
reduced Sobolev spaces and limiting embeddings, our arguments work
directly with very weak solutions and rely on Green representation,
pointwise Green-function estimates, concentration of $L^1$-data, and
potential theory. In particular, they remain applicable on general
domains even when the polyharmonic Green function changes sign.

\subsection{Main ideas of the proof and critical asymptotics}

Let $\Omega\subset\mathds{R}^{2m}$ ($m \geq 1$) be a smooth bounded
domain. The Green representation formula plays a central role throughout the
paper. More precisely, we prove in \hyperref[solution2]{Lemma \ref{solution2}} that, for
every $f\in L^1(\Omega)$, the unique very weak solution of
\eqref{POMEGA} is given by
\begin{equation} \label{eq:intro:GreenRep}
u(x)
=
(G_\Omega\star f)(x)
:=
\int_\Omega G_\Omega(x,y)f(y)\,\mathrm{d}y
\qquad\text{for a.e. }x\in\Omega.
\end{equation}
Following the strategy underlying the original argument of
Brezis and Merle \cite{BM}, we combine the above Green representation formula
with Jensen's inequality. Indeed, if
\[
\mathrm{d}\mu_f(y):=
\frac{|f(y)|}{\|f\|_{L^1(\Omega)}}\,\mathrm{d}y,
\]
then $\mu_f$ is a probability measure. Hence, combining \eqref{eq:intro:GreenRep} with the convexity of the exponential function, Jensen's inequality, and Tonelli's theorem, we have
\begin{align}
\int_\Omega
\exp\left(
\frac{\alpha|u(x)|}{\|f\|_{L^1(\Omega)}}
\right)\,\mathrm{d}x
&\leq
\int_\Omega
\exp\left(
\alpha\int_\Omega |G_\Omega(x,y)|\,d\mu_f(y)
\right)\,\mathrm{d}x
\notag
\leq
\int_\Omega\int_\Omega
e^{\alpha|G_\Omega(x,y)|}\,\mathrm{d}\mu_f(y)\,\mathrm{d}x
\notag\\
&=
\int_\Omega J_{\Omega,\alpha}(y)\,\mathrm{d}\mu_f(y)
\leq
\sup_{y\in\Omega}J_{\Omega,\alpha}(y),
\label{Jensen upper bound}
\end{align}
where
\[
J_{\Omega,\alpha}(y):=
\int_\Omega e^{\alpha|G_\Omega(x,y)|}\,\mathrm{d}x.
\]
Thus, the proof of the Brezis--Merle inequality is reduced to
obtaining a uniform estimate for the one-pole profile $J_{\Omega,\alpha}$. Indeed, we will show that for every $\alpha \in (0, \Lambda_m)$ there exists $C_{\Omega, \alpha}>0$ such that
\begin{equation} \label{eq:S(alpha)dfn}
S_\Omega(\alpha):=\sup_{y\in\Omega} J_{\Omega,\alpha}(y)\leq C_{\Omega, \alpha},
\end{equation}
see \hyperref[exp uniform G integrability]{Theorem \ref{exp uniform G integrability}}. The uniform estimate \eqref{eq:S(alpha)dfn}, combined with the preceding Jensen argument \eqref{Jensen upper bound}, shows that the left-hand side of \eqref{sharp BM} is uniformly
bounded over all nonzero $L^1$-data for every
$\alpha\in(0,\Lambda_m)$. By contrast,
\hyperref[sharpness theorem]{Theorem \ref{sharpness theorem}} proves
that no finite bound uniform over all nonzero $L^1$-data can hold
when $\alpha\geq\Lambda_m$. What remains is to determine the
best possible right-hand side in \eqref{sharp BM}. Obtaining this optimal constant, and hence
the fully sharp form \eqref{sharp BM} of the Brezis--Merle inequality stated in
\hyperref[main sharp BM]{Theorem \ref{main sharp BM}}, leads naturally
to the associated maximizing problem

\begin{equation}\label{L1 extremal value}
S_{\Omega,L^1}(\alpha)
:=
\sup_{\substack{f\in L^1(\Omega)\\
\|f\|_{L^1(\Omega)}=1}}
\int_\Omega
e^{\alpha|(G_\Omega\star f)(x)|}\,\mathrm{d}x,
\qquad
0<\alpha<\Lambda_m.
\end{equation}
Taking the supremum over all unit $L^1$-data in \eqref{Jensen upper bound},
we obtain
$$
S_{\Omega,L^1}(\alpha)\leq\sup_{y\in\Omega}J_{\Omega,\alpha}(y)=S_\Omega(\alpha).
$$
The reverse inequality will be obtained in \hyperref[prop:L1-value-and-nonattainment]{Proposition \ref{prop:L1-value-and-nonattainment}}, therefore
\begin{equation} \label{intro-optimal-value2}
S_{\Omega,L^1}(\alpha)=\sup_{y\in\Omega}J_{\Omega,\alpha}(y)=S_\Omega(\alpha).
\end{equation}

\hyperref[prop:L1-value-and-nonattainment]{Proposition \ref{prop:L1-value-and-nonattainment}} further shows that $S_{\Omega,L^1}(\alpha)$ is
\emph{not} attained. Nevertheless, the value $S_\Omega(\alpha)$ is approached by
normalized data concentrating at a point, whose associated measures
converge weakly-$*$ to a Dirac mass. This motivates extending the
admissible class to the space $\mathcal P(\Omega)$ of positive Radon probability measures and defining
\begin{equation}\label{SP def}
S_{\Omega,\mathcal P}(\alpha)
:=
\sup_{\mu\in\mathcal{P}(\Omega)}
\int_\Omega e^{\alpha|(G_\Omega\star\mu)(x)|}\,\mathrm{d}x, \quad \text{ where } \quad (G_\Omega\star\mu)(x)
    :=
    \int_\Omega G_\Omega(x,y)\,\mathrm{d}\mu(y).
\end{equation}
\hyperref[prop:equality-of-extremal-values]{Proposition \ref{prop:equality-of-extremal-values}} shows that this extension does not change the optimal value and
\begin{equation} \label{intro-optimal-value}
S_{\Omega,\mathcal P}(\alpha)
=
S_{\Omega,L^1}(\alpha)
=
S_\Omega(\alpha).
\end{equation}

In view of the definition \eqref{eq:S(alpha)dfn} of the value $S_\Omega(\alpha)$, the analysis of the maximizing problem $S_{\Omega,\mathcal P}(\alpha)$ is thus reduced to the qualitative study of
$J_{\Omega,\alpha}$. In \hyperref[section5]{Section \ref{section5}}, we prove that $J_{\Omega,\alpha}$ is continuous on $\Omega$,
determine its behavior as the pole approaches the boundary, and show
that $J_{\Omega,\alpha}$ attains its maximum at an interior point.

In particular, for every $\alpha \in (0, \Lambda_m)$,
\[
S_\Omega(\alpha)
=
\max_{y\in\Omega}J_{\Omega,\alpha}(y),
\]
see \hyperref[Jalpha sup attained]{Proposition \ref{Jalpha sup attained}}. This, together with \eqref{intro-optimal-value2}, shows that
\[
\max_{y\in\Omega}
\int_\Omega e^{\alpha|G_\Omega(x,y)|}\,\mathrm{d}x
\]
is precisely the optimal right-hand side in \eqref{sharp BM}. Thus, \eqref{sharp BM} is sharp
both with respect to the range of the exponent $\alpha \in (0, \Lambda_m)$ and in its right-hand side, which completes the proof of \hyperref[main sharp BM]{Theorem \ref{main sharp BM}}.

Moreover, for every $\alpha \in (0, \Lambda_m)$, the set of maximum points of $J_{\Omega,\alpha}$
\begin{equation}\label{intro-maximizer-set}
K_{\Omega,\alpha}
:=
\left\{
y\in\Omega:
J_{\Omega,\alpha}(y)=\max_{z\in\Omega}J_{\Omega,\alpha}(z)
\right\}
\end{equation}
is a nonempty compact subset of $\Omega$ (see \hyperref[Jalpha sup attained]{Proposition \ref{Jalpha sup attained}}). Since
\[
J_{\Omega,\alpha}(y)
=
\int_\Omega e^{\alpha|G_\Omega(x,y)|}\,\mathrm{d}x
=
\int_\Omega
e^{\alpha|(G_\Omega\star\delta_y)(x)|}\,\mathrm{d}x \qquad \text{ for every } y \in \Omega,
\]
every point of $K_{\Omega,\alpha}$ gives rise to a maximizing   Dirac
measure for $S_{\Omega,\mathcal P}(\alpha)$. The following theorem shows that these are the only
extremal measures for $S_{\Omega,\mathcal P}(\alpha)$.

\begin{theorem}
\label{extremal characterization}
Let $\alpha\in(0,\Lambda_m)$. A measure
$\mu\in\mathcal P(\Omega)$ satisfies
\[
\int_\Omega
e^{\alpha |(G_\Omega\star\mu)(x)|}\,\mathrm{d}x
=
S_{\Omega,\mathcal P}(\alpha)
\]
if and only if
\(
\mu=\delta_y\)
 for some \(y\in K_{\Omega,\alpha}\). In particular, the set of probability measures attaining
$S_{\Omega,\mathcal P}(\alpha)$ is precisely
\(
\{\delta_y:y\in K_{\Omega,\alpha}\}\).
\end{theorem}

Although $S_{\Omega,L^1}(\alpha)$ is not attained, its value is
approached by normalized data concentrating at any point
$y\in K_{\Omega,\alpha}$ and $S_{\Omega,\mathcal P}(\alpha)$ is attained in $\mathcal{P}(\Omega)$ precisely
by the corresponding Dirac masses. The preceding characterization reduces the determination of the
optimal bound in \eqref{sharp BM} to the problem of locating the maximizers of
$J_{\Omega,\alpha}$. On the unit ball, Boggio's explicit Green function allows
us to express the optimal
upper bound in an explicit form.

\begin{theorem}\label{thm:intro-ball}
Let $B:=B(0,1)\subset\mathds{R}^{2m}$ and let
$\alpha\in(0,\Lambda_m)$. Then $J_{B,\alpha}$ is radial and strictly
decreasing with respect to the distance from the origin.
Consequently,
\(
K_{B,\alpha}=\{0\}
\),
and $\delta_0$ is the unique extremal measure.
\end{theorem}

\noindent
As a direct consequence of \hyperref[thm:intro-ball]{Theorem \ref{thm:intro-ball}}, we obtain
the following sharp form of the Brezis--Merle inequality on the unit ball.

\begin{corollary}\label{sharp BM ball} Let $f\in L^1(B)\setminus\{0\}$, and let $u\in L^1(B)$ be the
corresponding very weak solution of \eqref{POMEGA} with $\Omega=B$. Then, for every \(\alpha\in(0,\Lambda_m)\),

\begin{equation*}
    \int_{B}
\exp\left(
\frac{\alpha |u(x)|}{\|f\|_{L^1(B)}}
\right)\,\mathrm{d}x\leq |\mathds{S}^{2m-1}|\int_{0}^{1}r^{2m-1-\alpha c_m}\exp\left[\Theta_{\alpha,m}(r)\right]\,\mathrm{d}r,
\end{equation*}
where
\begin{equation}
\label{Boggio correction profile}
\Theta_{\alpha,m}(r)
:=
\begin{cases}
0,
& \text{if } m=1,\\[0.6em]
\displaystyle
\alpha c_m
\sum_{k=1}^{m-1}
\frac{(-1)^k}{2k}
\binom{m-1}{k}
\left(1-r^{2k}\right),
& \text{if } m\geq 2.
\end{cases}
\end{equation}
The range $0<\alpha<\Lambda_m$ is sharp, and the right-hand side is
optimal for each fixed $\alpha$ in this range.
\end{corollary}

In the particular case $m=1$,
\hyperref[sharp BM ball]{Corollary \ref{sharp BM ball}} yields
\[
\int_B
\exp\left(
\frac{\alpha|u(x)|}{\|f\|_{L^1(B)}}
\right)\,\mathrm{d}x
\leq
\frac{4\pi^2}{4\pi-\alpha},
\qquad
0<\alpha<4\pi,
\]
see \hyperref[rmk:Fontana-Morpurgo]{Remark \ref{rmk:Fontana-Morpurgo}}. Since $|B|=\pi$, this is precisely
\eqref{eq:fontana-morpurgo-bound} with $\Omega=B$. Thus,
\hyperref[sharp BM ball]{Corollary \ref{sharp BM ball}} extends the result of Fontana and
Morpurgo \cite[Corollary 2]{FontanaMorpurgo} from the Laplacian to the polyharmonic operator of arbitrary
order in the corresponding conformal dimension.

Having characterized the extremal measures for
$S_{\Omega,\mathcal P}(\alpha)$, with $\alpha\in(0,\Lambda_m)$, we finally
investigate the behavior of the optimal value $S_\Omega(\alpha)$ and of the
maximizing set $K_{\Omega,\alpha}$ as $\alpha\uparrow\Lambda_m$. To formulate
the critical asymptotics, we introduce the Robin function
\begin{equation}\label{eq:robin-function}
R_\Omega(y)
:=
\lim_{x\to y}
\left[
G_\Omega(x,y)
-
c_m\log\left(\frac{1}{|x-y|}\right)
\right],\quad \text{ where }\quad c_m:=
    \frac{1}{2^{2m-1}\pi^m(m-1)!}.
\end{equation}

\begin{theorem}
\label{intro-critical}
We have
\[
\lim_{\alpha\uparrow\Lambda_m}(\Lambda_m-\alpha)S_\Omega(\alpha)
=
2^{2m}\pi^{2m}
\exp\left(
\Lambda_m\max_{y\in\Omega}R_\Omega(y)
\right).
\]
Moreover, let $(\alpha_k)\subset(0,\Lambda_m)$ satisfy
$\alpha_k\uparrow\Lambda_m$ as $k \to + \infty$, and choose
$y_{\alpha_k}\in K_{\Omega,\alpha_k}$ for every $k$. Then there exists $y^\ast\in\Omega$ such that, up to a
subsequence,
\[
y_{\alpha_k}\longrightarrow y^\ast,
\qquad  
R_\Omega(y^\ast)=\max_{y\in\Omega}R_\Omega(y), \qquad \text{ and } \qquad \delta_{y_{\alpha_k}}
\rightharpoonup^\ast
\delta_{y^\ast}
\qquad\text{in }\mathcal{M}(\Omega).
\]
If the Robin function has a unique maximum point, the convergence
holds without passing to a subsequence.
\end{theorem}

\noindent
The first conclusion determines the leading-order behavior of the optimal constant on the right-hand side of the Brezis--Merle inequality, while the second shows that the Robin function selects the possible limiting locations of the extremal measures for $S_{\Omega,\mathcal P}(\alpha_k)$ as $k \to + \infty$. Our approach is related in spirit to
Flucher's analysis of concentration for the Trudinger--Moser
inequality \cite{Flucher}. We refer to \hyperref[Flucher]{Remark \ref{Flucher}} for further details on this connection.

\begin{remark}
Although our results are formulated for polyharmonic operators of arbitrary integer orders $m \geq 1$,
several of their conclusions are new already in the case $m=1$, which corresponds to the Laplacian operator. To
the best of our knowledge, these include the non-existence of maximizers for 
$S_{\Omega,L^1}(\alpha)$ in $L^1(\Omega)$, the characterization of all
probability measures attaining $S_{\Omega,\mathcal P}(\alpha)$, and, as
$\alpha\uparrow4\pi$, the selection of the possible limits of the
maximizing points by the maximum set of the Robin function.
\end{remark}

\subsection{Outline of the paper and notation}
The paper is organized as follows. In \hyperref[section2]{Section \ref{section2}}, we develop the Green function framework for very weak solutions. In particular, we recall the pointwise bound for the polyharmonic Green function due to
Mayboroda and Maz'ya \cite{Mazya}, see \eqref{sharp Green function estimate}, which is one of the main ingredients of our arguments. Furthermore, we establish a crucial interior estimate for polyharmonic functions, stated in
\hyperref[interior estimate BMOMEGA]{Lemma \ref{interior estimate BMOMEGA}}. These tools yield the uniform pointwise
Green function estimates established in \hyperref[lem:uniform-green-estimates]{Lemma \ref{lem:uniform-green-estimates}}, while preserving
the exact coefficient $c_m$ of the logarithmic singularity (see \eqref{Fundamental Solution}). In \hyperref[section3-]{Section \ref{section3-}}, we prove the Brezis--Merle-type estimate together with the sharpness of its critical exponent. \hyperref[section3]{Section \ref{section3}} is devoted to the analysis of the associated optimal constant, including the non-existence of maximizers in $L^1(\Omega)$ and the extension to probability measures. 

In \hyperref[section5]{Section \ref{section5}}, we establish further properties of the one-pole profile $J_{\Omega,\alpha}$, which allow us, in particular, to identify the optimal constant on the right-hand side of \eqref{sharp BM} and complete the proof of \hyperref[main sharp BM]{Theorem \ref{main sharp BM}}.

\hyperref[section6]{Section \ref{section6}} is devoted to the characterization of all extremals for $S_{\Omega,\mathcal P}(\alpha)$ on general domains (which gives the proof of \hyperref[extremal characterization]{Theorem \ref{extremal characterization}}) and on the unit ball (which gives the proofs of \hyperref[thm:intro-ball]{Theorem \ref{thm:intro-ball}} and \hyperref[sharp BM ball]{Corollary \ref{sharp BM ball}}).

Finally, in \hyperref[section4]{Section \ref{section4}}, we determine the critical asymptotics of the optimal right-hand side of \eqref{sharp BM} and describe the selection of the limiting extremal measures by the polyharmonic Robin function.

Throughout the paper, we use the following notation. Unless otherwise stated, $m \geq 1$ is an integer and $\Omega\subset\mathds{R}^{2m}$ is a smooth bounded
domain. We denote by $D>0$ the diameter of $\Omega$, namely
\begin{equation} \label{eq:notation:diameter}
    D:=\mathrm{diam}(\Omega)= \sup \left\{|x-y| \: : \: x,y \in \Omega \right\},
\end{equation}
and by $d(z)>0$ the distance of a point $z \in \Omega$ from the boundary $\partial \Omega$, namely
\[
    d(z):=\operatorname{dist}(z,\partial\Omega)= \inf \left\{ |z-y| \: : \: y \in \partial \Omega \right\}.
\]
A ball centered at a point $x_0$ with radius $R>0$ is denoted by $B(x_0,R)$, and if $x_0=(0, \dots, 0)$ we simply write $B_R$.

\section{Green representation formula and Green function estimates}\label{section2}

We begin this section by justifying the Green representation formula \eqref{eq:intro:GreenRep} for the solutions of \eqref{POMEGA} in the sense of \eqref{very weak solution2}.

Let $\Omega \subset\mathds{R}^{2m}$ ($m \geq 1$) be a smooth bounded domain. We first introduce the Green function $G_{\Omega}$  of $(-\Delta)^m$ in $\Omega$ under
homogeneous Dirichlet boundary conditions. The fundamental solution of the polyharmonic operator $(-\Delta)^m$
in $\mathds R^{2m}$ is determined up to an additive constant and we take
\begin{equation}\label{Fundamental Solution}
    \Phi(x):=c_m\log\frac{D}{|x|},
\end{equation}
where $c_m:=[2^{2m-1}\pi^m(m-1)!]^{-1}$
and $D>0$ is the diameter of $\Omega$, see \eqref{eq:notation:diameter}.
It is well known that
\(
    (-\Delta)^{m}\Phi=\delta_0\) in
    \(\mathds{R}^{2m}
\), in the sense of distributions, namely
\[
    \int_{\mathds{R}^{2m}}\Phi(x)\,(-\Delta)^{m}\varphi(x)\,\mathrm{d}x=\varphi(0) \qquad \text{ for every } \varphi \in C_0^{\infty}(\mathds{R}^{2m}).
\]

The coefficient $c_m$ of the logarithmic singularity determines the
critical range of the exponent appearing in \eqref{sharp BM}. Indeed, for every
$\alpha>0$,
\begin{equation*}
    e^{\alpha\Phi(x)}
    =
    \left(\frac{D}{|x|}\right)^{\alpha c_m},
\end{equation*}
and this function is integrable in a neighbourhood of the origin if and only if
\(\alpha c_m<2m\). Since
\begin{equation}\label{eq:critical-constants}
    \Lambda_m=
    2^{2m}\pi^m m!=
    \frac{2m}{c_m} \quad \text{ and } \quad c_m \Lambda_m=2m,
\end{equation}
the condition
\(
    \alpha c_m<2m
\)
is equivalent to
\(
    \alpha<\Lambda_m\).
Although this is only a local integrability criterion for the fundamental solution, it already motivates the appearance of the critical exponent $\Lambda_m$ in \hyperref[main sharp BM]{Theorem \ref{main sharp BM}}. The corresponding pointwise estimates for $G_\Omega$, uniform with
respect to the pole, are established in
\hyperref[lem:uniform-green-estimates]{Lemma
\ref{lem:uniform-green-estimates}}.

Now, fix $x\in\Omega$.
By translation invariance of $(-\Delta)^m$, we have $(-\Delta)_{y}^{m}\Phi(y-x)=\delta_x$ in $\mathds{R}^{2m}$, in the sense of distributions; in particular, 
\begin{equation}\label{fundamental property}
    \int_{\Omega} \Phi(y-x)\,(-\Delta)^{m}\varphi(y)\,\mathrm{d}y
    =\varphi(x)
    \qquad \text{ for every } \varphi \in C_0^{\infty}(\Omega).
\end{equation}
To take into account the boundary conditions in \eqref{POMEGA}, we introduce $H_{\Omega,x}$ as the solution
of the auxiliary problem

\begin{equation}\label{auxiliary}
\begin{cases}
(-\Delta)^m H_{\Omega,x}(y) = 0 & \text{in } \Omega,\\[4pt]
\partial_\nu^{j-1} H_{\Omega,x}(y) = \partial_\nu^{j-1}\Phi(x-y)\ \text{for }j=1,\dots,m & \text{on } \partial\Omega.
\end{cases}
\end{equation}
Whenever \eqref{auxiliary} admits a classical solution, we define the Green function of $(-\Delta)^m$ in $\Omega$ with Dirichlet boundary conditions as
\[
G_{\Omega}(x,y)\coloneqq \Phi(x-y)-H_{\Omega,x}(y).
\]
By construction,
\(
    (-\Delta)_y^{m}G_{\Omega}(x,y)=\delta_x\) in \(\Omega\), in the sense of distributions,
and 
\(
    \partial_{\nu}^{j-1}G_{\Omega}(x,y)=0\) pointwise on $\partial\Omega$, for every \(j=1,\dots,m\). Moreover, the Green function is symmetric.

Problem \eqref{auxiliary} does admit a classical solution for every $x\in\Omega$ in our setting: Krasovski\u{\i} proved that the Green function of a smooth bounded domain exists and is unique (see \cite[Theorem~3.3]{ExistenceGreen}). Finally, for each fixed $x\in\Omega$, the regular part $H_{\Omega,x}$ of $G_{\Omega}$ is smooth in $\Omega$.

We are now in position to prove the distributional property of $G_{\Omega}$ with respect to the test space $C_{\mathrm{cl}}^{2m}(\overline{\Omega})$ introduced in \eqref{test space}.
\begin{lemma}\label{dist property2}
For every
$\varphi\in C_{\mathrm{cl}}^{2m}(\overline{\Omega})$
and every $x\in\Omega$, we have
\begin{equation}
\label{distribution property}
\varphi(x)
=
\int_\Omega
G_\Omega(x,y)(-\Delta)^m\varphi(y)\,\mathrm{d}y.
\end{equation}
\end{lemma}

\begin{proof}
Fix
$\varphi\in C_{\mathrm{cl}}^{2m}(\overline{\Omega})$
and consider the boundary value problem
\begin{equation}
\begin{cases}
(-\Delta)^m\psi=(-\Delta)^m\varphi
    & \text{in }\Omega,\\
\partial_\nu^j\psi=0
    & \text{on }\partial\Omega,
      \qquad j=0,\ldots,m-1.
\end{cases}
\end{equation}
Clearly, $\psi=\varphi$ is a classical solution of this problem. We
claim that it is the unique one. Indeed, if $\psi$ is another classical
solution and $w:=\psi-\varphi$, then
\[
\begin{cases}
(-\Delta)^m w=0
    & \text{in }\Omega,\\
\partial_\nu^j w=0
    & \text{on }\partial\Omega,
      \qquad j=0,\ldots,m-1.
\end{cases}
\]
By the uniqueness part of the $L^p$ theory for the polyharmonic
Dirichlet problem (see \cite[Theorem~2.20]{Gazzola}),
the homogeneous problem admits only the trivial strong solution.
Since $w$ is classical, it is in particular a strong solution.
Therefore, $w=0$, and hence $\psi=\varphi$.

Thus, $\varphi$ is the unique classical solution of the preceding
boundary value problem. Applying
\cite[Theorem~3.3 and (3.29)]{ExistenceGreen} gives
\eqref{distribution property}.
\end{proof}

The following pointwise bound for the polyharmonic Green function, due to
Mayboroda and Maz'ya \cite{Mazya}, will be crucial in our analysis:
there exists a constant $C_{\Omega}>0$ such that
\begin{equation}\label{sharp Green function estimate}
|G_{\Omega}(x,y)|
\leq
C_{\Omega}
\log\left(
1+\frac{\min\{d(x),d(y)\}}{|x-y|}
\right),
\qquad
x,y\in\Omega,\quad x\neq y.
\end{equation}
This is precisely \cite[Theorem~3.1]{Mazya} applied with $n=2m$ and $i=j=0$.

The above estimate \eqref{sharp Green function estimate}, combined with the identity \eqref{distribution property}, yields the following Green representation formula for very
weak solutions of \eqref{POMEGA}.

\begin{lemma}\label{solution2}
For every $f\in L^1(\Omega)$, problem \eqref{POMEGA} admits a unique
very weak solution in the sense of \eqref{very weak solution2}.
Moreover, this solution is represented by
\begin{equation}\label{Green representation}
u(x)=\int_\Omega G_\Omega(x,y)f(y)\,\mathrm{d}y
\qquad\text{for a.e. }x\in\Omega.
\end{equation}
\end{lemma}
\begin{proof} We first show that the function defined by the right-hand side of
\eqref{Green representation} is well defined almost everywhere and
belongs to $L^1(\Omega)$. Since $d(x),d(y)\leq D$, \eqref{sharp Green function estimate} yields 
\[
|G_{\Omega}(x,y)|
\leq
C_{\Omega}
\log\left(
1+\frac{D}{|x-y|}
\right),
\qquad
x,y\in\Omega,\quad x\neq y.
\]
Consequently, for each fixed $y\in\Omega$, the change of variables
$z=x-y$ gives
\begin{align*}
\int_{\Omega}|G_{\Omega}(x,y)|\,\mathrm{d}x
&\leq
C_{\Omega}
\int_{\Omega-y}
\log\left(
1+\frac{D}{|z|}
\right)\,\mathrm{d}z \leq
C_{\Omega}
\int_{B_D}
\log\left(
1+\frac{D}{|z|}
\right)\,\mathrm{d}z\\
&=
C_{\Omega}|\mathds{S}^{2m-1}|
\int_0^D
r^{2m-1}
\log\left(
1+\frac{D}{r}
\right)\,\mathrm{d}r
=:C_0.
\end{align*}
The last integral is finite, namely $C_0 \in (0,+ \infty)$. Indeed, after the change of variables
$r=Ds$, we have
\[
\int_0^D
r^{2m-1}
\log\left(
1+\frac{D}{r}
\right)\,\mathrm{d}r
=
D^{2m}
\int_0^1
s^{2m-1}
\log\left(
1+\frac{1}{s}
\right)\,\mathrm{d}s
<+\infty.
\]
Moreover, the constant $C_0$ is independent of $y$. Therefore,
\[
\sup_{y\in\Omega}
\int_{\Omega}|G_{\Omega}(x,y)|\,\mathrm{d}x
\leq C_0.
\]
It now follows from Tonelli's theorem that
\begin{equation*}
\int_{\Omega}\int_{\Omega}
|G_{\Omega}(x,y)|\,|f(y)|\,\mathrm{d}y\,\mathrm{d}x
=
\int_{\Omega}
|f(y)|
\left(
\int_{\Omega}|G_{\Omega}(x,y)|\,\mathrm{d}x
\right)\mathrm{d}y
\leq
C_0\|f\|_{L^1(\Omega)}
<+\infty.
\end{equation*}
Hence, the function 
\[
u(x):=
\int_{\Omega}G_{\Omega}(x,y)f(y)\,\mathrm{d}y
\]
is well defined for a.e. $x\in\Omega$. Furthermore,
\begin{equation*}
\|u\|_{L^1(\Omega)}
=
\int_{\Omega}
\left|
\int_{\Omega}G_{\Omega}(x,y)f(y)\,\mathrm{d}y
\right|\mathrm{d}x
\leq
\int_{\Omega}\int_{\Omega}
|G_{\Omega}(x,y)|\,|f(y)|\,\mathrm{d}y\,\mathrm{d}x
\leq
C_0\|f\|_{L^1(\Omega)}.
\end{equation*}
Thus, $u\in L^1(\Omega)$.

We next verify that $u$ is a very weak solution of
\eqref{POMEGA}. Indeed, for each $\varphi\in C_{\mathrm{cl}}^{2m}(\overline{\Omega}),$ by the preceding estimate,
\begin{align*}
&\int_{\Omega}\int_{\Omega}
\left|
G_{\Omega}(x,y)f(y)(-\Delta)^m\varphi(x)
\right|\,\mathrm{d}y\,\mathrm{d}x
\leq
\|(-\Delta)^m\varphi\|_{L^{\infty}(\Omega)}
\int_{\Omega}\int_{\Omega}
|G_{\Omega}(x,y)|\,|f(y)|\,\mathrm{d}y\,\mathrm{d}x\\
&\qquad\leq
C_0
\|(-\Delta)^m\varphi\|_{L^{\infty}(\Omega)}
\|f\|_{L^1(\Omega)}
<+\infty.
\end{align*}
Therefore, Fubini's theorem, the symmetry of the Green function,
and \hyperref[dist property2]{Lemma \ref{dist property2}} yield
\begin{align*}
&\int_{\Omega}
u(x)(-\Delta)^m\varphi(x)\,\mathrm{d}x
=
\int_{\Omega}\int_{\Omega}
G_{\Omega}(x,y)f(y)(-\Delta)^m\varphi(x)
\,\mathrm{d}y\,\mathrm{d}x\\
& \qquad=
\int_{\Omega}
f(y)
\left(
\int_{\Omega}
G_{\Omega}(x,y)(-\Delta)^m\varphi(x)\,\mathrm{d}x
\right)\mathrm{d}y
=
\int_{\Omega}f(y)\varphi(y)\,\mathrm{d}y.
\end{align*}
Thus, $u$ is a very weak solution of \eqref{POMEGA}.

To establish uniqueness, suppose that there exists another very weak solution $v \in L^1(\Omega)$ of \eqref{POMEGA}. This implies that $w=u-v$ satisfies
\begin{equation}\label{uniquesol}
    \int_{\Omega} w(-\Delta)^m\varphi\,\mathrm{d}x=0, \qquad \text{ for every } \varphi \in C_{\mathrm{cl}}^{2m}(\overline{\Omega}).
\end{equation}
We consider $H_0^{m}(\Omega)$ endowed with the inner product defined by
\begin{equation*}
\langle \xi,\eta\rangle_{H_0^{m}(\Omega)}
:=
\begin{cases}
\displaystyle
\int_\Omega
(-\Delta)^k\xi\,(-\Delta)^k\eta\,\mathrm{d}x,
& \text{if } m=2k,\\[1.2em]
\displaystyle
\int_\Omega
\nabla(-\Delta)^k\xi\cdot
\nabla(-\Delta)^k\eta\,\mathrm{d}x,
& \text{if } m=2k+1.
\end{cases}
\end{equation*}
Moreover, given $\psi\in C_0^\infty(\Omega)$, we consider the linear
functional
\begin{equation*}
\mathcal{L}_\psi(\eta)
:=
\int_\Omega \psi\eta\,\mathrm{d}x,
\qquad
\eta\in H_0^m(\Omega).
\end{equation*}
By the Cauchy--Schwarz and Poincar\'e inequalities, we have
\begin{equation*}
\left|\mathcal{L}_\psi(\eta)\right|
\leq
\|\psi\|_{L^2(\Omega)}
\|\eta\|_{L^2(\Omega)}
\leq
C\|\psi\|_{L^2(\Omega)}
\|\eta\|_{H_0^{m}(\Omega)}.
\end{equation*}
Thus, $\mathcal{L}_\psi$ is a bounded linear functional on
$H_0^m(\Omega)$. By the Riesz representation theorem, there exists
a unique $\zeta_\psi\in H_0^m(\Omega)$ such that
\begin{equation}\label{Riesz polyharmonic}
\langle \zeta_\psi,\eta\rangle_{H_0^{m}(\Omega)}
=
\int_\Omega \psi\eta\,\mathrm{d}x
\qquad
 \text{ for every } \eta\in H_0^m(\Omega).
\end{equation}
Equivalently, $\zeta_\psi$ is the weak solution of
\begin{equation}\label{Riesz representation problem}
\left\{
\begin{aligned}
(-\Delta)^m\zeta_\psi&=\psi
&&\text{in }\Omega,\\
\partial_\nu^j\zeta_\psi&=0
&&\text{on }\partial\Omega,
\qquad j=0,\ldots,m-1,
\end{aligned}
\right.
\end{equation}
and, since $\Omega$ is smooth, standard elliptic estimates (see \cite[Section 2.5.2]{Gazzola}) imply  that $\zeta_{\psi} \in C^{\infty}(\overline{\Omega})$. In this way $\zeta_{\psi} \in C^{\infty}(\overline{\Omega})\cap H_0^{m}(\Omega)$, proving in particular that $\zeta_{\psi} \in C_{\mathrm{cl}}^{2m}(\overline{\Omega})$, that is, $\zeta_{\psi}$ is a classical solution of \eqref{Riesz representation problem}. Then, from \eqref{uniquesol}, it follows
\begin{equation*}
    0=\int_{\Omega} w(-\Delta)^m \zeta_{\psi}\,\mathrm{d}x=\int_{\Omega}w \psi\,\mathrm{d}x\qquad \text{ for every } \psi \in C_0^{\infty}(\Omega),
\end{equation*}
implying by the fundamental lemma of the calculus of variations (see \cite[Corollary 4.24.]{BrezisFA}) that $w=0$ almost everywhere, that is, $u=v$ in $L^1(\Omega).$
\end{proof}

Next, we recall Boggio's formula for the Green function in a ball. Following Boggio \cite{Boggioclassical}, see also \cite[Lemma 2.27]{Gazzola}, the Green function for the
polyharmonic operator \((-\Delta)^m\) in the ball \(B(x_0,R)\subset
\mathds{R}^{2m}\), with Dirichlet boundary conditions
\(
\partial_\nu^j u=0\) on \(\partial B(x_0,R)\),
\(j=0,1,\ldots,m-1\),
is the positive map $G_{B(x_0,R)}:\overline{B(x_0,R)}\times\overline{B(x_0,R)}
\to \mathds{R}\cup\{+\infty\}$ given, for \(x\neq y\), by Boggio's formula
\begin{equation}\label{Boggio integral formula}
G_{B(x_0,R)}(x,y)
=
c_m
\int_{1}^{\frac{[XY]}{|x-y|}}
\frac{(t^2-1)^{m-1}}{t^{2m-1}}\,\mathrm{d}t,
\end{equation}
where
\[
[XY]\coloneq
\frac{1}{R}
\sqrt{
|x-x_0|^2|y-x_0|^2
-2R^2\langle x-x_0,y-x_0\rangle
+R^4
}.
\]
For $m\geq2$, expanding the integrand in
\eqref{Boggio integral formula} and integrating term by term yields
\begin{equation}\label{eq:Boggio-explicit}
G_{B(x_0,R)}(x,y)
=
c_m
\left[
\log\left(\frac{[XY]}{|x-y|}\right)
+
\sum_{k=1}^{m-1}
\frac{(-1)^k}{2k}\binom{m-1}{k}
\left(
1-\left(\frac{|x-y|}{[XY]}\right)^{2k}
\right)
\right].
\end{equation}
\begin{remark}\label{case m=1 remark}
When $m=1$, the integral representation
\eqref{Boggio integral formula} reduces to
\begin{equation} \label{eq:Boggio-explicit1}
G_{B(x_0,R)}(x,y)
=
c_1
\int_1^{[XY]/|x-y|}
\frac{\mathrm{d}t}{t}
=
c_1\log\left(\frac{[XY]}{|x-y|}\right).
\end{equation}
This gives a meaning to the right-hand side of \eqref{eq:Boggio-explicit} also when $m=1$. Throughout the paper, we will use \eqref{eq:Boggio-explicit} for every $m\geq1$ (in particular, including the case $m=1$, which has to be understood in the sense of \eqref{eq:Boggio-explicit1}).
\end{remark}
When \(y=x_0\), we obtain
\begin{equation}\label{Green especial}
G_{B(x_0,R)}(x,x_0)
=
c_m
\left[
\log\left(\frac{R}{|x-x_0|}\right)
+
\sum_{k=1}^{m-1}
\frac{(-1)^k}{2k}\binom{m-1}{k}
\left(
1-\frac{|x-x_0|^{2k}}{R^{2k}}
\right)
\right].
\end{equation}
In order to derive Green-function estimates that are uniform with
respect to the pole, we shall use the following interior-to-annular
estimate for polyharmonic functions. In the harmonic case, such an
estimate follows immediately from the maximum principle. For
higher-order operators, for which no general maximum principle is
available, it will play an analogous role.

\begin{lemma} \label{interior estimate BMOMEGA}
Let \(m\geq 1\), \(R>0\), \(B(x_0,R)\subset\mathds{R}^n\), with $n \geq 2$, and let
\(v\in C^\infty(B(x_0,R))\) be such that
\[
(-\Delta)^m v=0
\quad\text{in }B(x_0,R).
\]
Then there exists a positive constant \(C=C(n,m)\) such that
\begin{equation} \label{eq:interior estimate BMOMEGA}
\|v\|_{L^\infty(B(x_0,R/2))}
\leq
C(n,m)
\|v\|_{L^\infty(B(x_0,R)\setminus B(x_0,R/2))}.
\end{equation}
\end{lemma}
\begin{proof}
First, we establish an elementary estimate that will be used in the
iteration argument to prove \eqref{eq:interior estimate BMOMEGA}. Let $\rho>0$ and suppose that
\(
z\in C^2(B_\rho)\cap C(\overline{B_\rho})
\)
satisfies
\(
|\Delta z|\leq K\) in
\(B_\rho
\)
for some $K\geq0$. We claim that
\begin{equation}\label{internal-barrier-estimate}
\|z\|_{L^\infty(B_\rho)}
\leq
\|z\|_{L^\infty(\partial B_\rho)}
+
\frac{\rho^2}{2n}K.
\end{equation}
Indeed, to simplify notation, let
\(
M_{\partial}
:=
\|z\|_{L^\infty(\partial B_\rho)}
\)
and let us define
\[
q(x):=
\frac{\rho^2-|x|^2}{2n},
\qquad x\in B_\rho.
\]
Then
\[
q\geq0
\quad\text{in }B_\rho,
\qquad
q=0
\quad\text{on }\partial B_\rho,
\qquad
\Delta q=-1
\quad\text{in }B_\rho.
\]
Consider
\(
\Psi_+:=z-M_{\partial}-Kq
\).
Since $\Delta z\geq-K$ in $B_\rho$, we obtain
\(
\Delta\Psi_+
=
\Delta z+K
\geq0\) in \(B_\rho
\). Moreover, by the definition of $M_\partial$,
\(
\Psi_+\leq0
\) on \(\partial B_\rho
\).
The classical maximum principle therefore gives
\(
\Psi_+\leq0\) in \(B_\rho\),
and hence
\[
z\leq M_{\partial}+Kq
\qquad\text{in }B_\rho.
\]
Similarly, consider
\(
\Psi_-:=-z-M_{\partial}-Kq
\).
Since $\Delta z\leq K$ in $B_\rho$, it follows that
\(
\Delta\Psi_-
=
-\Delta z+K
\geq0
\) in \(B_\rho\). Furthermore,
\(
\Psi_-\leq0\) on \(\partial B_\rho\).
Another application of the maximum principle gives
\[
-z\leq M_{\partial}+Kq
\qquad\text{in }B_\rho.
\]
Combining the preceding inequalities, we find
\[
|z(x)|
\leq
M_{\partial}+Kq(x)
\leq
M_{\partial}+\frac{\rho^2}{2n}K,
\qquad x\in B_\rho.
\]
This proves \eqref{internal-barrier-estimate}.

Next, we return to the proof of \eqref{eq:interior estimate BMOMEGA}. In view of the invariance properties of polyharmonic functions and the $L^\infty$-norm with respect to translations and dilations of the spatial variable, it suffices to prove \eqref{eq:interior estimate BMOMEGA} in the case $x_0=0$ and $R=1$.

Set
\[
M:=
\|v\|_{L^\infty(B_1\setminus B_{1/2})}.
\]
If $M=+\infty$, there is nothing to prove. Therefore, we may
assume that $M<+\infty$.

Consider the annulus
\[
A:=
\left\{
x\in\mathds{R}^{n}:
\frac{2}{3}<|x|<\frac{5}{6}
\right\}.
\]
Since
\(
\partial B_{3/4}
\subset\subset
A
\subset\subset
B_1\setminus B_{1/2}\),
we have
\[
\|v\|_{L^\infty(A)}
\leq
\|v\|_{L^\infty(B_1\setminus B_{1/2})}
=
M.
\]
Fix $p:=n+1$ and choose $\eta>0$ such that
\(
B(y,2\eta)\subset\subset A\) for every
\(y\in\partial B_{3/4}\).
Since $(-\Delta)^m v=0$ in $B_1$ and $\eta>0$ is fixed, the
local estimates for elliptic equations of order $2m$ (\cite[Theorem 2.20]{Gazzola}) yield
\[
\|v\|_{W^{2m,p}(B(y,\eta))}
\leq
C(n,m)
\|v\|_{L^1(B(y,2\eta))}
\quad
\forall\,y\in\partial B_{3/4}.
\]
Here, the constant $C$ depends on the fixed radius $\eta$ but is independent of the center $y$. Moreover, since
\(
B(y,2\eta)\subseteq A\subseteq B_1\setminus B_{1/2}\),
we obtain
\begin{equation*}
\|v\|_{W^{2m,p}(B(y,\eta))}
\leq
C(n,m)
\|v\|_{L^1(B(y,2\eta))}
\leq
C(n,m)|B_{2\eta}|
\|v\|_{L^\infty(B(y,2\eta))}
\leq
C(n,m)M.
\end{equation*}
Since $p=n+1>n$, Morrey's inequality gives
\[
W^{2m,p}(B(y,\eta))
\hookrightarrow
C^{2m-1,\gamma}(\overline{B(y,\eta)}),
\qquad
\gamma:=1-\frac{n}{p}>0.
\]
Therefore,
\[
\|v\|_{C^{2m-1,\gamma}(\overline{B(y,\eta)})}
\leq
C(n,m)M
\quad
\forall\,y\in\partial B_{3/4}.
\]
For every $j=0,1,\ldots,m-1$, the function
$(-\Delta)^jv$ is a linear combination of derivatives of $v$
of order $2j$, and $2j\leq 2m-2$. Consequently,
\[
|(-\Delta)^jv(y)|
\leq
C(n,m)M,
\qquad
y\in\partial B_{3/4},
\qquad
j=0,1,\ldots,m-1.
\]
Taking the supremum over $y\in\partial B_{3/4}$, we conclude that
\begin{equation}\label{boundary-estimate-polyharmonic}
\|(-\Delta)^jv\|_{L^\infty(\partial B_{3/4})}
\leq
C(n,m)M,
\qquad
j=0,1,\ldots,m-1.
\end{equation}

Finally, we prove \eqref{eq:interior estimate BMOMEGA} via an iteration argument, applying the above estimates \eqref{internal-barrier-estimate} and \eqref{boundary-estimate-polyharmonic} recursively. In what follows, $C>0$ denotes a
constant which may change from line to line and depends only on
$n$ and $m$. Define
\[
w_j:=(-\Delta)^jv,
\qquad
j=0,1,\ldots,m-1.
\]
Since $v\in C^\infty(B_1)$, each $w_j$ belongs to $C^\infty(\overline{B_{1/2}})$. Moreover,
\[
-\Delta w_j=w_{j+1} \quad \text{ in } \quad B_1,
\qquad
j=0,1,\ldots,m-2,
\]
and
\[
-\Delta w_{m-1}
=
(-\Delta)^m v
=
0
\qquad\text{in }B_1.
\]
Thus, in particular, $w_{m-1}$ is harmonic in $B_{3/4}$. By the classical
maximum principle and \eqref{boundary-estimate-polyharmonic},
\begin{equation} \label{eq:int-est-hyp}
\|w_{m-1}\|_{L^\infty(B_{3/4})}
\leq
\|w_{m-1}\|_{L^\infty(\partial B_{3/4})}
\leq
CM.
\end{equation}

If $m=1$, then $w_{m-1}=w_0=v$, and therefore
\[
\|v\|_{L^\infty(B_{1/2})} \leq \|v\|_{L^\infty(B_{3/4})}
\leq CM,
\]
which is the desired estimate \eqref{eq:interior estimate BMOMEGA}.

From now on we assume that $m\geq2$. Let
\(
j\in\{0,1,\ldots,m-2\}
\)
and suppose that
\begin{equation} \label{eq:int-est-rep-hyp}
\|w_{j+1}\|_{L^\infty(B_{3/4})}
\leq
CM.
\end{equation}
Since
\(
-\Delta w_j=w_{j+1}\) in
 \(B_{3/4}
\),
we have
\(
|\Delta w_j|
=
|w_{j+1}|
\leq
CM\) in \(B_{3/4}\).
Applying \eqref{internal-barrier-estimate} with $z=w_j$ and
$\rho=3/4$, together with \eqref{boundary-estimate-polyharmonic} and the assumption \eqref{eq:int-est-rep-hyp},  we obtain
\[
\|w_j\|_{L^\infty(B_{3/4})}
\leq
\|w_j\|_{L^\infty(\partial B_{3/4})}
+
\frac{(3/4)^2}{2n}
\|w_{j+1}\|_{L^\infty(B_{3/4})}\\
\leq
CM.
\]
We complete the proof of \eqref{eq:interior estimate BMOMEGA} by descending induction, since the above argument can be repeated at each step. Indeed, for $j=m-2$, \eqref{eq:int-est-hyp} holds; hence \eqref{eq:int-est-rep-hyp} is satisfied and the above argument yields
\[
\|w_{m-2}\|_{L^\infty(B_{3/4})}
\leq
CM.
\]
Iterating this procedure down to $j=0$, we conclude that
\[
\|w_j\|_{L^\infty(B_{3/4})}
\leq
CM,
\qquad
j=0,1,\ldots,m-1.
\]
In particular, since $w_0=v$,
\[
\|v\|_{L^\infty(B_{1/2})}
\leq\|v\|_{L^\infty(B_{3/4})}
\leq
CM.
\]

\end{proof}

We complete this section by establishing the Green-function estimates that will be used
throughout the paper. The first two estimates are uniform with respect
to the pole, whereas the last one concerns the behavior near each
fixed pole. As a direct consequence of the latter,
Green profiles corresponding to distinct poles can be separated
in a neighborhood of either pole.

\begin{lemma}\label{lem:uniform-green-estimates}
Let $\Omega\subset\mathds{R}^{2m}$ be a smooth bounded domain. There exists a constant $C=C(m,\Omega)>0$ such that, for every
$y\in\Omega$,

\begin{equation}\label{faraway pole}
|G_\Omega(x,y)|
\leq C
\quad
\forall\,
x\in
\Omega\setminus B\left(y,\frac{d(y)}4\right),
\end{equation}
and
\begin{equation}\label{Green omega to ball}
|G_\Omega(x,y)|
\leq
C+c_m\log\frac{d(y)/2}{|x-y|}
\quad
\forall\,
x\in B\left(y,\frac{d(y)}4\right)\setminus\{y\}.
\end{equation}

Moreover, for every fixed $y\in\Omega$, there exist constants
$r_y\in(0,d(y)/4)$ and $C_y>0$ such that
\begin{equation}\label{inferior green estimate}
G_\Omega(x,y)
\geq
c_m\log\frac{D}{|x-y|}-C_y>0
\quad
\forall\,
x\in B(y,r_y)\setminus\{y\}.
\end{equation}
In particular,
\begin{equation} \label{eq:GreenFunctLarge}
G_\Omega(x,y)\longrightarrow+\infty
\qquad\text{as }x\longrightarrow y.
\end{equation}
Consequently, if $y_0,y_1\in\Omega$ are distinct, then there
exists $\rho>0$ such that
\begin{equation}
\label{Green pole separation}
|G_\Omega(x,y_1)|>|G_\Omega(x,y_0)|
\quad
\forall\,
x\in B(y_1,\rho)\setminus\{y_1\}.
\end{equation}
\end{lemma}
\begin{proof} We begin with the proof of \eqref{faraway pole}. Let $y \in \Omega$. We observe that $|x-y|\geq d(y)/4$ for every $x \in \Omega\setminus B(y,d(y)/4),$ and directly from \eqref{sharp Green function estimate} we see that $|G_{\Omega}(x,y)|\leq C_{\Omega}\log 5$ for every $x \in \Omega\setminus B(y,d(y)/4)$.

Let us now prove \eqref{Green omega to ball}. To this end, let us fix \(y\in\Omega\) and define
\begin{equation*}
    W_y(x)\coloneq G_{\Omega}(x,y)-G_{B(y,d(y)/2)}(x,y), \quad x \in B(y,d(y)/2).
\end{equation*}

We claim that $(-\Delta)^{m}_x W_y(x)=0$ in $B(y,d(y)/2)$ in the classical sense, for each $y \in \Omega$. Indeed, considering $\varphi \in C_0^{\infty}(B(y,d(y)/2))$, by the construction and symmetry of the Green function, we know that
\begin{equation}\label{GreenBall}
    \int_{B(y,d(y)/2)} G_{B(y,d(y)/2)}(x,y)(-\Delta)^{m}\varphi(x)\,\mathrm{d}x = \varphi(y).
\end{equation}
Extending \(\varphi\) by zero outside \(B(y,d(y)/2)\), we obtain \(\varphi\in C_0^\infty(\Omega)\). Thus, again by the symmetry of the Green function,
\begin{equation}\label{GreenOmega}
    \int_{B(y,d(y)/2)} G_{\Omega}(x,y)(-\Delta)^m\varphi(x)\,\mathrm{d}x=\int_{\Omega} G_{\Omega}(x,y)(-\Delta)^m\varphi(x)\,\mathrm{d}x=\varphi(y).
\end{equation}
Now, subtracting \eqref{GreenBall} from \eqref{GreenOmega}, we conclude that 
\begin{equation*}
    \int_{B(y,d(y)/2)} W_y(x)(-\Delta)^m\varphi(x)\,\mathrm{d}x=0,
\end{equation*}
that is, $(-\Delta)^{m}_x W_y(x)=0$ in $B(y,d(y)/2)$ in the sense of distributions. Next, since $(-\Delta)^m$ is an elliptic operator with $C^{\infty}$ coefficients, by  \cite[Corollary 6.34]{FollandPDE}, we see that $(-\Delta)^m$ is a hypoelliptic operator, implying that $W_y \in C^{\infty}(B(y,d(y)/2))$ for each $y \in \Omega$, concluding the proof of our claim.

We also claim that $\lVert W_y\rVert_{L^{\infty}(A_y)}\leq C$, where $A_y\coloneq B(y,d(y)/2)\setminus B(y,d(y)/4)$ and $C$ is independent of $y \in \Omega$. Indeed, from \eqref{faraway pole}, we already know that $|G_{\Omega}(x,y)|\leq C$ in $A_{y} \subseteq \Omega \setminus B(y, d(y)/4)$. We can now estimate
$G_{B(y,d(y)/2)}(x,y)$ in $A_y.$ From the definition of $A_y$,
for every $x\in A_y$, we have
\begin{equation*}
\frac{d(y)}{4}
\leq
|x-y|
<
\frac{d(y)}{2}, \quad \text{ or equivalently } \quad
\frac{1}{2}
\leq
\frac{2|x-y|}{d(y)}
<1.
\end{equation*}
Consequently,
\begin{equation*}
0
<
\log\frac{d(y)/2}{|x-y|}
\leq
\log 2,
\end{equation*}
and, for every $k=1,\ldots,m-1$,
\begin{equation*}
0
<
1-
\left(
\frac{2|x-y|}{d(y)}
\right)^{2k}
\leq
1-\frac{1}{2^{2k}}
\leq 1.
\end{equation*}
Therefore, by \eqref{Green especial} and the positivity of Boggio's
Green function, we obtain
$$
0
\leq
G_{B(y,d(y)/2)}(x,y)
\leq
c_m
\left[
\log 2
+
\sum_{k=1}^{m-1}
\frac{1}{2k}
\binom{m-1}{k}
\right]
=:C(m),
\qquad x\in A_y.
$$
In particular,
\begin{equation*}
\left\|
G_{B(y,d(y)/2)}(\cdot,y)
\right\|_{L^\infty(A_y)}
\leq
C(m),
\end{equation*}
where the constant $C(m)$ is independent of $y\in\Omega$,
from which our second claim follows. Since $W_y$ is polyharmonic in $B(y,d(y)/2)$ for each $y \in \Omega$, we can combine \hyperref[interior estimate BMOMEGA]{Lemma \ref{interior estimate BMOMEGA}} with the preceding claim to obtain
\begin{equation*}
    \lVert W_y\rVert_{L^{\infty}(B(y,d(y)/4))}\leq C\lVert W_y\rVert_{L^{\infty}(A_y)}\leq C.
\end{equation*}
Finally, for every $x \in B(y,d(y)/4)$ we have
$$
    |G_{\Omega}(x,y)|\leq G_{B(y,d(y)/2)}(x,y)+|W_y(x)|
\leq c_m\log \frac{d(y)/2}{|x-y|}+C.
$$

Next, we prove \eqref{inferior green estimate}. Using the symmetry of Green's function, that is $G_{\Omega}(x,y)=G_{\Omega}(y,x)$, we can write 
$$
    G_{\Omega}(x,y)=\Phi(x-y)-H_{\Omega,y}(x)
    =c_m\log\frac{D}{|x-y|}-H_{\Omega,y}(x),
$$
where \(H_{\Omega,y}\) denotes the regular part of the Green function. More precisely, for each fixed \(y\in\Omega\), the map \(x\mapsto H_{\Omega,y}(x)\) is smooth in \(\Omega\) and solves the corresponding problem \eqref{auxiliary}.

Choose $0<r_0<\min\{D,d(y)\}$ so that $\overline{B(y,r_0)}\subseteq\Omega$. Since
$H_{\Omega,y}$ is continuous on $\overline{B(y,r_0)}$, there exists
$C_y>0$ such that
\[
|H_{\Omega,y}(x)|\leq C_y
\quad
\forall\,x\in B(y,r_0).
\]
We may then choose
\[
r_y\in\left(0,\min\left\{r_0,\frac{d(y)}4\right\}\right)
\]
sufficiently small so that
\[
c_m\log\frac{D}{r_y}>C_y.
\]
Consequently, for every $x\in B(y,r_y)\setminus\{y\}$,
\[
G_\Omega(x,y)
\geq
c_m\log\frac{D}{|x-y|}-C_y
>0.
\]
Since the logarithmic term tends to $+\infty$ as $x\to y$, it
also follows that
\[
G_\Omega(x,y)\longrightarrow+\infty
\qquad\text{as }x\longrightarrow y.
\]
Finally, let $y_0,y_1\in\Omega$ be distinct and choose $\rho_0>0$
such that $\overline{B(y_1,\rho_0)}\subset\subset\Omega$ and
$y_0\notin\overline{B(y_1,\rho_0)}$. Since the map
$x\mapsto G_\Omega(x,y_0)$ is continuous on
$\overline{B(y_1,\rho_0)}$, there exists $C>0$ such that
$|G_\Omega(x,y_0)|\leq C$ for every $x\in B(y_1,\rho_0)$. On the
other hand, $G_\Omega(x,y_1)\to+\infty$ as $x\to y_1$. Hence there
exists $\rho\in(0,\rho_0)$ such that
$|G_\Omega(x,y_1)|>C\geq|G_\Omega(x,y_0)|$ for every
$x\in B(y_1,\rho)\setminus\{y_1\}$. This proves
\eqref{Green pole separation}.
\end{proof}
\section{The Brezis--Merle inequality in the $L^1$ setting and sharpness of its critical exponent} \label{section3-}

The purpose of this section is to prove that the left-hand side in \eqref{sharp BM} is uniformly bounded for any exponent $\alpha$ in the range $(0,\Lambda_m)$ and to establish the sharpness of the exponent $\Lambda_m$; the optimality of the right-hand side in \eqref{sharp BM} will be addressed in \hyperref[section3]{Section \ref{section3}}. In view of the Jensen reduction \eqref{Jensen upper bound}, the main task is then to bound $J_{\Omega,\alpha}$ uniformly with respect to its pole. This is achieved by comparing $G_\Omega(\cdot,y)$ locally with the Green function of a ball centered at $y$, while controlling the resulting polyharmonic remainder. Finally, the logarithmic singularity of $G_\Omega$, together with $L^1$-data concentrating at a point, shows that the parameter range in \eqref{sharp BM} cannot be enlarged.

Let $f\in L^1(\Omega)\setminus\{0\}$, and let $u$ be the corresponding
very weak solution of \eqref{POMEGA}.  Combining the Green representation \eqref{Green representation} with \eqref{Jensen upper bound}, the desired uniform boundedness reduces to proving that for every 
$\alpha\in(0,\Lambda_m)$,
\begin{equation*}
\sup_{y\in\Omega}J_{\Omega,\alpha}(y)
=
\sup_{y\in\Omega}
\int_\Omega e^{\alpha|G_\Omega(x,y)|}\,\mathrm{d}x
<+\infty.
\end{equation*}
This is an immediate consequence of the next theorem, whose proof relies on
\hyperref[lem:uniform-green-estimates]{Lemma \ref{lem:uniform-green-estimates}}.

\begin{theorem}\label{exp uniform G integrability}
There exists a constant $C=C(m,\Omega)>0$ such that, for every
$\alpha\in(0,\Lambda_m)$ and every $y\in\Omega$, we have
\begin{equation}\label{int exp near pole}
\int_{B(y,d(y)/4)}
e^{\alpha|G_\Omega(x,y)|}\,\mathrm{d}x
\leq
C\frac{d(y)^{2m}}{\Lambda_m-\alpha}
\end{equation}
and
\begin{equation}\label{int exp faraway pole}
\int_{\Omega\setminus B(y,d(y)/4)}
e^{\alpha|G_\Omega(x,y)|}\,\mathrm{d}x
\leq C.
\end{equation}
Consequently, for every
$\alpha\in(0,\Lambda_m)$ and every $y\in\Omega$, we have
\begin{equation}\label{Jalpha and distance to the boundary}
J_{\Omega,\alpha}(y)
=
\int_\Omega e^{\alpha|G_\Omega(x,y)|}\,\mathrm{d}x
\leq
C+C\frac{d(y)^{2m}}{\Lambda_m-\alpha}.
\end{equation}
In particular, for every fixed $\alpha\in(0,\Lambda_m)$,
\begin{equation}\label{uniform-J-bound}
\sup_{y\in\Omega}J_{\Omega,\alpha}(y)<+\infty.
\end{equation}
\end{theorem}
\begin{proof}
Fix $y \in \Omega$. By \eqref{Green omega to ball}, one has
\[
e^{\alpha|G_\Omega(x,y)|}
\leq
e^{\alpha C}
\left(\frac{d(y)/2}{|x-y|}\right)^{\alpha c_m}
\qquad
\forall x \in B(y,d(y)/4) \setminus \{y\}.
\]
Therefore
\begin{align*}
\int_{B(y,d(y)/4)}
e^{\alpha|G_\Omega(x,y)|}\,\mathrm{d}x
&\leq
C d(y)^{\alpha c_m}
\int_0^{d(y)/4} r^{2m-1-\alpha c_m}\,\mathrm{d}r\\
&=
C d(y)^{\alpha c_m}
\frac{(d(y)/4)^{2m-\alpha c_m}}
     {2m-\alpha c_m}
\leq
C
\frac{
d(y)^{2m}
}{
\Lambda_m-\alpha
},
\end{align*}
where we also used \eqref{eq:critical-constants}, which gives
\(
2m-\alpha c_m=c_m(\Lambda_m-\alpha)>0
\)
and 
\(
4^{-(2m-\alpha c_m)}\leq1
\).

On the other hand, \eqref{faraway pole}  and $\alpha<\Lambda_m$ give
\[
\int_{\Omega\setminus B(y,d(y)/4)}
e^{\alpha|G_\Omega(x,y)|}\,\mathrm{d}x
\leq
e^{\Lambda_m C}|\Omega|.
\]

The above two estimates yield \eqref{int exp near pole} and \eqref{int exp faraway pole}, respectively. Combining them, we obtain the asserted bound \eqref{Jalpha and distance to the boundary} for
$J_{\Omega,\alpha}(y)$. Finally, since $d(y)\leq D$, the uniform boundedness \eqref{uniform-J-bound}
of $J_{\Omega,\alpha}$ follows for every fixed
$\alpha\in(0,\Lambda_m)$.
\end{proof}

Combining
\hyperref[exp uniform G integrability]{Theorem
\ref{exp uniform G integrability}} with the Jensen estimate
\eqref{Jensen upper bound}, we conclude that the left-hand side of
\eqref{sharp BM} is uniformly bounded throughout the subcritical
range $(0,\Lambda_m)$. This proves the subcritical boundedness asserted in
\hyperref[main sharp BM]{Theorem \ref{main sharp BM}}.

\begin{remark}
For each fixed datum $f\in L^1(\Omega)$, the corresponding very weak
solution $u$ of \eqref{POMEGA} satisfies
\[
e^{k|u|}\in L^1(\Omega)
\quad\forall\,k>0.
\]
Indeed, by the truncation argument used in Corollary~1 of
\cite{BM}, one may decompose $f$ into a bounded part and a remainder
with arbitrarily small $L^1$-norm, and then apply the preceding
subcritical estimate to the solution corresponding to the
remainder.
\end{remark}

Next, we prove that the critical exponent $\Lambda_m$ in
\eqref{sharp BM} cannot be enlarged. The argument combines the
divergence of $J_{\Omega,\alpha}(y)$ for $\alpha\geq\Lambda_m$ with normalized
$L^1$-data concentrating at $y$.
\begin{lemma}\label{32pi2 sharpness lemma} For any $\alpha \geq \Lambda_m$ and $y \in \Omega$, we have $J_{\Omega,\alpha}(y)=+\infty.$
\end{lemma}
\begin{proof}
Fix $y\in\Omega$, and let
$r_y\in(0,d(y)/4)$ and $C_y>0$ as in \eqref{inferior green estimate}. In view of \eqref{inferior green estimate}, 
for every $\alpha\geq\Lambda_m$, we have
$$
J_{\Omega,\alpha}(y)
\geq
\int_{B(y,r_y)}e^{\alpha|G_\Omega(x,y)|}\,\mathrm{d}x\geq
e^{-\alpha C_y}D^{\alpha c_m}
\int_{B(y,r_y)}|x-y|^{-\alpha c_m}\,\mathrm{d}x=
C\int_0^{r_y}r^{2m-1-\alpha c_m}\,\mathrm{d}r
=+\infty
$$
because $2m-1-\alpha c_m\leq-1$ if and only if
$\alpha\geq\Lambda_m$.
\end{proof}
\begin{theorem}\label{sharpness theorem} The exponent $\Lambda_m$ that appears in \eqref{sharp BM} is sharp.
\end{theorem}
\begin{proof}
Suppose, by contradiction, that there exist
$\alpha\geq\Lambda_m$ and $C_{\Omega,\alpha}>0$ such that the
left-hand side of \eqref{sharp BM} is bounded above by
$C_{\Omega,\alpha}$ for every very weak solution of
\eqref{POMEGA} corresponding to a datum
$f\in L^1(\Omega)\setminus\{0\}$. Fix $y \in \Omega$ and, for every sufficiently large $k\in\mathds N$,
define
\begin{equation}\label{eq:concentrating-data}
    f_{k,y}(x)\coloneq \frac{\chi_{B(y,\sfrac{1}{k})}(x)}{|B(y,\sfrac{1}{k})|}.
\end{equation}
Clearly $f_{k,y} \geq 0$ and $(f_{k,y}) \subseteq L^1(\Omega)$ with $\lVert f_{k,y} \rVert_{L^1(\Omega)}=1$ for every sufficiently large $k \in \mathds{N}$. From \hyperref[solution2]{Lemma \ref{solution2}}, for each $k \in \mathds{N}$, the unique very weak solution of \eqref{POMEGA} with $f=f_{k,y}$ can be written as
\begin{equation*}
    u_{k,y}(x)\coloneq \int_{\Omega}G_{\Omega}(x,z)f_{k,y}(z)\,\mathrm{d}z.
\end{equation*}
Now, by the continuity of $G_\Omega$ far away from the diagonal,
\begin{equation}\label{eq:concentrating-potential}
    \lim_{k\to+\infty}u_{k,y}(x)=\lim_{k\to+\infty}\frac{1}{|B(y,1/k)|}\int_{B(y,1/k)} G_{\Omega}(x,z)\,\mathrm{d}z=G_{\Omega}(x,y)\quad \forall\, x \in \Omega \setminus\{y\}.
\end{equation}
Thus, by Fatou's lemma and the assumed uniform bound, we have
\begin{align*}
    J_{\Omega,\alpha}(y)=\int_{\Omega} e^{\alpha|G_{\Omega}(x,y)|}\,\mathrm{d}x=\int_{\Omega} \liminf_{k \to +\infty} e^{\alpha |u_{k,y}(x)|}\,\mathrm{d}x\leq \liminf_{k\to+\infty} \int_{\Omega} e^{\alpha|u_{k,y}(x)|}\,\mathrm{d}x\leq C_{\Omega,\alpha},
\end{align*}
which directly contradicts \hyperref[32pi2 sharpness lemma]{Lemma \ref{32pi2 sharpness lemma}}.
\end{proof}

\section{The optimal constant $S_{\Omega,L^1}(\alpha)$ in the Brezis--Merle inequality}\label{section3}
Fix $\alpha\in(0,\Lambda_m)$. In this section, we prove the non-existence of maximizers for $S_{\Omega,L^1}(\alpha)$ (see \hyperref[prop:L1-value-and-nonattainment]{Proposition
\ref{prop:L1-value-and-nonattainment}}) and the optimal-value
identity \eqref{intro-optimal-value}. At this
stage, \eqref{Jensen upper bound} and \eqref{uniform-J-bound} give
\[
S_{\Omega,L^1}(\alpha)
\leq
\sup_{y\in\Omega}J_{\Omega,\alpha}(y)
<+\infty.
\]
On the other hand, for each $y\in\Omega$, the concentrating family $(f_{k,y})$ defined in
\eqref{eq:concentrating-data} suggests that the reverse inequality also holds.
These observations lead us to expect that
\begin{equation}\label{eq:L1-supJ}
S_{\Omega,L^1}(\alpha)=\sup_{y\in\Omega}J_{\Omega,\alpha}(y).
\end{equation}
The proposition
below establishes this identity rigorously by means of
concentrating sequences. The non-existence of maximizers for $S_{\Omega,L^1}(\alpha)$ is then
obtained from the equality case in the Jensen reduction
\eqref{Jensen upper bound} and the pole-separation property
\eqref{Green pole separation}.

\begin{proposition}\label{prop:L1-value-and-nonattainment}
For every $\alpha\in(0,\Lambda_m)$, \eqref{eq:L1-supJ} holds. Moreover, the supremum defining $S_{\Omega,L^1}(\alpha)$ is not attained
by any function in $L^1(\Omega)$.
\end{proposition}

\begin{proof}
The Jensen estimate \eqref{Jensen upper bound} shows that, for every
$f\in L^1(\Omega)$ satisfying
$\lVert f\rVert_{L^1(\Omega)}=1$,
\begin{equation}\label{eq:L1-upper-bound-by-J}
    \int_\Omega
    \exp\left(
        \alpha\left|(G_\Omega\star f)(x)\right|
    \right)
    \,\mathrm{d}x
    \leq
    \int_\Omega J_{\Omega,\alpha}(y)|f(y)|\,\mathrm{d}y
    \leq
    \sup_{y\in\Omega}J_{\Omega,\alpha}(y).
\end{equation}
Consequently,
\begin{equation*}
    S_{\Omega,L^1}(\alpha)
    \leq
    \sup_{y\in\Omega}J_{\Omega,\alpha}(y).
\end{equation*}

For the reverse inequality, fix $y\in\Omega$ and consider the
sequence $(f_{k,y})$ defined in \eqref{eq:concentrating-data}.
Then $\|f_{k,y}\|_{L^1(\Omega)}=1$ for every sufficiently large
$k\in\mathds N$, and, by
\eqref{eq:concentrating-potential}, 
\(
    (G_\Omega\star f_{k,y})(x)
    \rightarrow G_\Omega(x,y) \)
for every \(x\in\Omega\setminus\{y\}\).
 Hence, by Fatou's lemma,
\begin{equation*}
    J_{\Omega,\alpha}(y)
    \leq
    \liminf_{k\to+\infty}
    \int_\Omega
    \exp\left(
        \alpha\left|(G_\Omega\star f_{k,y})(x)\right|
    \right)
    \,\mathrm{d}x
    \leq
    S_{\Omega,L^1}(\alpha).
\end{equation*}
Since $y\in\Omega$ is arbitrary, we conclude the proof of \eqref{eq:L1-supJ}.

It remains to prove that the supremum $S_{\Omega,L^1}(\alpha)$ is not attained. Suppose, by
contradiction, that there exists $f\in L^1(\Omega)$ such that
\(
    \lVert f\rVert_{L^1(\Omega)}=1
\)
and
\begin{equation*}
    \int_\Omega
    \exp\left(
        \alpha\left|(G_\Omega\star f)(x)\right|
    \right)
    \,\mathrm{d}x
    =
    S_{\Omega,L^1}(\alpha).
\end{equation*}
Since $f$ is assumed to be extremal and
\eqref{eq:L1-supJ} has already been proved, the first and last
terms in the Jensen reduction \eqref{Jensen upper bound}
coincide. Therefore, equality holds at every intermediate step
of that chain. We now make explicit the consequence for Jensen's
inequality. For a.e. $x\in\Omega$,
define its Jensen defect by
\begin{equation*}
\mathcal D_f(x)
:={}
\int_\Omega
e^{\alpha|G_\Omega(x,y)|}
|f(y)|\,\mathrm{d}y
-
\exp\left(
\alpha
\int_\Omega
|G_\Omega(x,y)|\,|f(y)|\,\mathrm{d}y
\right).
\end{equation*}
The equality between the second and third terms in \eqref{Jensen upper bound} gives
\[
\int_\Omega \mathcal D_f(x)\,\mathrm{d}x=0.
\]
Moreover, by Jensen's inequality,
\(
\mathcal D_f(x)\geq0\) for a.e. \(x\in\Omega\). Hence, it follows that
\[
\mathcal D_f(x)=0
\qquad\text{for a.e. }x\in\Omega.
\]
Therefore, there exists a measurable set $X_0\subseteq\Omega$ such
that $|\Omega\setminus X_0|=0$ and, for every $x\in X_0$,
\begin{equation*}
\exp\left(
\alpha
\int_\Omega
|G_\Omega(x,y)|\,|f(y)|\,\mathrm{d}y
\right)
=
\int_\Omega
e^{\alpha|G_\Omega(x,y)|}
|f(y)|\,\mathrm{d}y.
\end{equation*}
Thus, equality holds in Jensen's inequality for every $x\in X_0$.
Since the exponential function is strictly convex, the equality
case in Jensen's inequality implies that, for every $x\in X_0$,
there exists a constant $c_x\geq0$ such that
\begin{equation*}
    |G_\Omega(x,y)|=c_x
\end{equation*}
for $|f(y)|\,\mathrm{d}y$-a.e. $y\in\Omega$ (see Theorem 2.2 of \cite{Lieb}).

Since $\Omega$ is open and $X_0\subseteq\Omega$ is such
that $|\Omega\setminus X_0|=0$, the set $X_0$ is dense in $\Omega$. We may therefore choose a countable dense set
\begin{equation*}
    \{x_n:n\in\mathds{N}\}
    \subseteq X_0.
\end{equation*}

For every $n\in\mathds{N}$, there exist $c_n\geq0$
and a measurable set $E_n\subseteq\Omega\setminus\{x_n\}$ such that
\begin{equation}\label{eq:Green-constant-En}
    |G_\Omega(x_n,y)|=c_n
    \quad
    \forall\,y\in E_n
\end{equation}
and
\begin{equation}\label{eq:En-full-f-mass}
    \int_{\Omega\setminus E_n}|f(y)|\,\mathrm{d}y=0.
\end{equation}
Define
\(
    E
    :=
    \bigcap_{n=1}^{\infty}E_n\).
Since
\(
    \Omega\setminus E
    =
    \bigcup_{n=1}^{\infty}(\Omega\setminus E_n)\),
it follows from \eqref{eq:En-full-f-mass} that
\begin{equation*}
    \int_{\Omega\setminus E}|f(y)|\,\mathrm{d}y=0.
\end{equation*}
Consequently,
\begin{equation*}
    \int_E|f(y)|\,\mathrm{d}y
    =
    \int_\Omega|f(y)|\,\mathrm{d}y
    =
    1.
\end{equation*}

In particular, $E$ has positive Lebesgue measure, and $E$ must contain at least two distinct points $y_0$, $y_1 \in E$ with $y_0 \neq y_1$.

Since $y_0,y_1\in E_n$ for every $n\in\mathds{N}$,
\eqref{eq:Green-constant-En} gives
\begin{equation}\label{eq:Green-equality-dense-set}
    |G_\Omega(x_n,y_0)|
    =
    c_n
    =
    |G_\Omega(x_n,y_1)|
    \quad
    \forall\,n\in\mathds{N}.
\end{equation}
Now fix $x\in\Omega\setminus\{y_0,y_1\}$. Since
$\{x_n:n\in\mathds{N}\}$ is dense in $\Omega$, there exists a
subsequence $(x_{n_k})$ such that
\(
    x_{n_k}\longrightarrow x\) as $k \to + \infty$.
The functions
    \(z\mapsto G_\Omega(z,y_0)\)
and
    \(z\longmapsto G_\Omega(z,y_1)\)
are continuous at $x$. Passing to the limit in
\eqref{eq:Green-equality-dense-set}, we obtain
\begin{equation}\label{eq:two-poles-equality}
    |G_\Omega(x,y_0)|
    =
    |G_\Omega(x,y_1)|
    \quad
    \forall\,
    x\in\Omega\setminus\{y_0,y_1\}.
\end{equation}
Since $y_0\neq y_1$, identity
\eqref{eq:two-poles-equality} contradicts
\eqref{Green pole separation}. Therefore,
$S_{\Omega,L^1}(\alpha)$ is not attained in $L^1(\Omega)$.
\end{proof}

For every $y\in\Omega$, the measures $f_{k,y}(z)\,\mathrm{d}z$, with
$f_{k,y}$ defined in \eqref{eq:concentrating-data}, converge
weakly-$*$ to $\delta_y$, which is not induced by any
$L^1(\Omega)$-density.
Together with \hyperref[prop:L1-value-and-nonattainment]{Proposition
\ref{prop:L1-value-and-nonattainment}}, this motivates extending the admissible class
from normalized $L^1$-data to positive Radon probability measures.
It suffices to consider positive measures, since
\eqref{eq:L1-upper-bound-by-J} depends only on $|f|$ and the functions
$f_{k,y}$ are nonnegative.

We denote by
\begin{equation*}
    \mathcal{P}(\Omega)
    :=
    \left\{
        \mu\in\mathcal{M}^{+}(\Omega):
        \mu(\Omega)=1
    \right\}
\end{equation*}
the set of positive Radon probability measures on $\Omega$. For
$\mu\in\mathcal{P}(\Omega)$, we define
\begin{equation*}
    (G_\Omega\star\mu)(x)
    :=
    \int_\Omega G_\Omega(x,y)\,\mathrm{d}\mu(y).
\end{equation*}
This potential is well defined for a.e. $x\in\Omega$.
Indeed, using the elementary inequality
    \(t\leq\frac{1}{\alpha}e^{\alpha t}\), \(t\geq0\),
 Tonelli's theorem, and \hyperref[exp uniform G integrability]{Theorem \ref{exp uniform G integrability}}, we find
\begin{align*}
    \int_\Omega\int_\Omega
    |G_\Omega(x,y)|
    \,\mathrm{d}\mu(y)\,\mathrm{d}x
    \leq
    \frac{1}{\alpha}
    \int_\Omega\int_\Omega
    e^{\alpha|G_\Omega(x,y)|}
    \,\mathrm{d}\mu(y)\,\mathrm{d}x
    &=
    \frac{1}{\alpha}
    \int_\Omega
    J_{\Omega,\alpha}(y)\,\mathrm{d}\mu(y)\\
    &\leq
    \frac{1}{\alpha}
    \sup_{y\in\Omega}J_{\Omega,\alpha}(y)
    <+\infty.
\end{align*}
Consequently,
\begin{equation*}
    \int_\Omega
    |G_\Omega(x,y)|
    \,\mathrm{d}\mu(y)
    <+\infty \quad \text{ for a.e. } x \in \Omega.
\end{equation*}

The following proposition establishes uniform exponential integrability
over $\mathcal P(\Omega)$ for every
$\alpha\in(0,\Lambda_m)$ and shows that this range of the exponent is
sharp. The optimal right-hand side and all equality cases are
identified in \hyperref[extremal characterization]{Theorem
\ref{extremal characterization}}.

\begin{proposition}\label{cor:BM-probability-measures}
For every $\alpha\in(0,\Lambda_m)$, there exists a constant
$C_{\Omega, \alpha}>0$ such that
\begin{equation*}
    \int_\Omega
    \exp\left(
        \alpha\left|(G_\Omega\star\mu)(x)\right|
    \right)
    \,\mathrm{d}x
    \leq
    C_{\Omega, \alpha}
\end{equation*}
for every $\mu\in\mathcal{P}(\Omega)$. Moreover, the range
$\alpha\in(0,\Lambda_m)$ is sharp.
\end{proposition}

\begin{proof}
Let $\mu\in\mathcal P(\Omega)$. Repeating the Jensen reduction
\eqref{Jensen upper bound}, with $\mu$ in place of $\mu_f$, and
using Tonelli's theorem, we obtain
\begin{equation*}
\int_\Omega e^{\alpha|(G_\Omega\star\mu)(x)|}\,\mathrm{d}x
\leq
\int_\Omega J_{\Omega,\alpha}(y)\,\mathrm{d}\mu(y)
\leq
\sup_{y\in\Omega}J_{\Omega,\alpha}(y)
\leq C_{\Omega, \alpha},
\end{equation*}
where the last inequality follows from \hyperref[exp uniform G integrability]{Theorem \ref{exp uniform G integrability}}.

To prove the sharpness, fix $y\in\Omega$ and consider
$\mu=\delta_y$. If $\alpha\geq\Lambda_m$, then
\begin{equation*}
    \int_\Omega
    \exp\left(
        \alpha\left|(G_\Omega\star\delta_y)(x)\right|
    \right)
    \,\mathrm{d}x
    =
    J_{\Omega,\alpha}(y)
    =
    +\infty
\end{equation*}
by \hyperref[32pi2 sharpness lemma]{Lemma \ref{32pi2 sharpness lemma}}. Therefore, no finite bound uniform over all
$\mu\in\mathcal P(\Omega)$ can hold when
$\alpha\geq\Lambda_m$.
\end{proof}

\hyperref[cor:BM-probability-measures]{Proposition \ref{cor:BM-probability-measures}} shows that
$S_{\Omega,\mathcal P}(\alpha)$ is finite, where $S_{\Omega,\mathcal P}(\alpha)$ is defined in \eqref{SP def}. We now prove that passing from
normalized $L^1$-data to probability measures does not change the
optimal value, thereby establishing the first equality in
\eqref{intro-optimal-value}.

\begin{proposition}\label{prop:equality-of-extremal-values}
For every $\alpha\in(0,\Lambda_m)$, one has
\begin{equation*}
    S_{\Omega,\mathcal P}(\alpha)
    =
    S_{\Omega,L^1}(\alpha)
    =
    \sup_{y\in\Omega}J_{\Omega,\alpha}(y).
\end{equation*}
\end{proposition}

\begin{proof}
\hyperref[cor:BM-probability-measures]{Proposition \ref{cor:BM-probability-measures}} gives
\begin{equation*}
    S_{\Omega,\mathcal P}(\alpha)
    \leq
    \sup_{y\in\Omega}J_{\Omega,\alpha}(y).
\end{equation*}
Conversely, for every $y\in\Omega$, the Dirac mass $\delta_y$
belongs to $\mathcal{P}(\Omega)$ and satisfies
\(
    (G_\Omega\star\delta_y)(x)
    =
    G_\Omega(x,y)
\)
for a.e. $x\in\Omega$. Hence,

\begin{equation*}
    J_{\Omega,\alpha}(y) = \int_\Omega
    \exp\left(
        \alpha\left|(G_\Omega\star\delta_y)(x)\right|
    \right)
    \,\mathrm{d}x \leq S_{\Omega,\mathcal P}(\alpha)
    .
\end{equation*}
Taking the supremum over $y\in\Omega$, we obtain
\begin{equation*}
    \sup_{y\in\Omega}J_{\Omega,\alpha}(y)
    \leq
    S_{\Omega,\mathcal P}(\alpha).
\end{equation*}

Therefore,
\(
    S_{\Omega,\mathcal P}(\alpha)
    =
    \sup_{y\in\Omega}J_{\Omega,\alpha}(y)\). Combining this identity with
\hyperref[prop:L1-value-and-nonattainment]{Proposition \ref{prop:L1-value-and-nonattainment}}, we conclude  the proof.
\end{proof}

\section{Properties of the one-pole profile \(J_{\Omega,\alpha}\)}\label{section5}

Let $\alpha \in (0, \Lambda_m)$. \hyperref[prop:equality-of-extremal-values]{Proposition \ref{prop:equality-of-extremal-values}} and \hyperref[prop:L1-value-and-nonattainment]{Proposition
\ref{prop:L1-value-and-nonattainment}} reduce the study of the optimal constant $S_{\Omega,L^1}(\alpha)$ in the Brezis--Merle inequality and the analysis of the attainability of $S_{\Omega,\mathcal P}(\alpha)$ to the one-point optimization problem
\[
\sup_{y\in\Omega}J_{\Omega,\alpha}(y).
\]
This motivates a more detailed
study of $J_{\Omega,\alpha}$. In this section, we  establish the properties of $J_{\Omega,\alpha}$ needed to prove that
$J_{\Omega,\alpha}$ attains its maximum in $\Omega$. To this end,  we first show that
$J_{\Omega,\alpha}$ is continuous on $\Omega$ and determine its limit as the pole
approaches $\partial\Omega$.

\begin{lemma}\label{Jalpha continuous}
For every \(\alpha\in(0,\Lambda_{m})\), the map $J_{\Omega,\alpha}:\Omega\to\mathds R$ is continuous.
\end{lemma}
\begin{proof} Fix $y_0 \in \Omega,$ and let $(y_n) \subseteq \Omega$ be a sequence such that $y_n \to y_0$. Our objective is to prove that $J_{\Omega,\alpha}(y_n)\to J_{\Omega,\alpha}(y_0).$ Let $0<r<d(y_0)/9$. We claim that there exists $n_0 \in \mathds{N}$ 
\begin{equation}\label{domain compatibility}
    n \geq n_0\quad \text{implies}\quad B(y_0,r) \subseteq B(y_n,d(y_n)/4).
\end{equation}
Indeed, since $y_n \to y_0$, we can find $n_0 \in \mathds{N}$ such that  $|y_n-y_0|<r$ for every $n \geq n_0$. Now let $x \in B(y_0,r)$, then we have for every $n \geq n_0$
\begin{equation*}
    |x-y_n|\leq |x-y_0|+|y_0-y_n|<2r.
\end{equation*}
Since for every $n \geq n_0$,
\begin{equation*}
    d(y_n) \geq d(y_0)-|y_0-y_n|>d(y_0)-r> 8r \quad \text{ and } \quad 2r<\frac{d(y_n)}{4},
\end{equation*}
the claim follows.

Next, write for $0<r<d(y_0)/9$
\begin{align*}
    |J_{\Omega,\alpha}(y_n)-J_{\Omega,\alpha}(y_0)|\leq&  \left |\int_{\Omega\setminus B(y_0,r)} e^{\alpha |G_{\Omega}(x,y_n)|}-e^{\alpha |G_{\Omega}(x,y_0)|}\,\mathrm{d}x\right|\\
    &+\int_{B(y_0,r)}e^{\alpha|G_{\Omega}(x,y_n)|}\,\mathrm{d}x+\int_{B(y_0,r)}e^{\alpha|G_{\Omega}(x,y_0)|}\,\mathrm{d}x\\
    & := J_{1,n}(r)+J_{2,n}(r)+J_{3,n}(r).
\end{align*}

We begin with the analysis of $J_{1,n}(r).$ Let $n_1 \in \mathds{N}$ be large enough such that $|y_n-y_0|<r/2$ for every $n \geq n_1$. If $x \in \Omega \setminus B(y_0,r)$ we see that $|x-y_0|\geq r$ and thus, for $n\geq n_1,$
\begin{equation*}
    |x-y_n|\geq |x-y_0|-|y_n-y_0|\geq r-\frac{r}{2}=\frac{r}{2}\quad \forall\, x \in \Omega\setminus B(y_0,r).
\end{equation*}
In this way, since the Green function $G_{\Omega}$ is continuous far away from the diagonal, we see that
\begin{equation}\label{Green pointwise}
    e^{\alpha |G_{\Omega}(x,y_n)|}\to e^{\alpha|G_{\Omega}(x,y_0)|}\quad \forall\,x \in \Omega\setminus B(y_0,r).
\end{equation}
Moreover, by \eqref{sharp Green function estimate}, there exists a positive constant $C(r)$ such that $|G_{\Omega}(x,y_n)|\leq C(r)$ for every $ x \in \Omega \setminus B(y_0,r)$ and $n \geq n_1.$ Thus, for $n \geq n_1$
\begin{equation}\label{Lebesgue boundness}
    e^{\alpha|G_{\Omega}(x,y_n)|}\leq e^{\alpha C(r)}\in L^1(\Omega \setminus B(y_0,r)).
\end{equation}
Next, from \eqref{Green pointwise} and \eqref{Lebesgue boundness}, we apply the Lebesgue dominated convergence theorem to conclude that
\begin{equation*}
    \int_{\Omega \setminus B(y_0,r)} e^{\alpha|G_{\Omega}(x,y_n)|}\,\mathrm{d}x\to \int_{\Omega \setminus B(y_0,r)} e^{\alpha|G_{\Omega}(x,y_0)|}\,\mathrm{d}x,
\end{equation*}
that is, $\lim_{n\to +\infty} J_{1,n}(r)=0$ for every $0<r<d(y_0)/9$. In fact, this argument works for every $r>0$ such that $B(y_0,r)\subset \subset \Omega$.

Let us now analyze $J_{2,n}.$ First, by \eqref{Green omega to ball}, after increasing $C$ if necessary, we have
\begin{equation*}
    |G_{\Omega}(x,y_n)|\leq C+c_m\log\frac{1}{|x-y_n|}\quad \forall\,x \in B(y_n,d(y_n)/4).
\end{equation*}
Moreover, for $n \geq n_0,$ \eqref{domain compatibility} holds, and thus, for every $n \geq n_0,$
\begin{equation*}
    |G_{\Omega}(x,y_n)|\leq C+c_m\log\frac{1}{|x-y_n|}\quad \forall\,x \in B(y_0,r).
\end{equation*}
Since $\alpha<\Lambda_m$, by
\eqref{eq:critical-constants}, we have
$\alpha c_m<\Lambda_m c_m=2m$; moreover,
$B(y_0,r)-y_n\subseteq B(0,2r)$. Therefore, for $n \geq n_0,$
\begin{align*}
    J_{2,n}(r) &\leq  C\int_{B(y_0,r)} \frac{1}{|x-y_n|^{\alpha c_m}}\,\mathrm{d}x
    \leq C \int_{B(0,2r)} \frac{1}{|z|^{\alpha c_m}}\,\mathrm{d}z
    = C \int_{0}^{2r} \delta^{2m-1-\alpha c_m}\,\mathrm{d}\delta
    = C r^{2m-\alpha c_m},
\end{align*}
where the constant $C>0$ above does not depend on $n \geq n_0$ and $r\in(0,d(y_0)/9)$.

Since $0<r<d(y_0)/9$, we clearly have $B(y_0,r)\subseteq B(y_0,d(y_0)/4)$ and consequently
\begin{equation*}
    |G_{\Omega}(x,y_0)|\leq C+c_m\log\frac{1}{|x-y_0|}\quad \forall\,x \in B(y_0,r).
\end{equation*}
Thus, proceeding as in the analysis of $J_{2,n}(r)$, we also obtain $J_{3,n}(r)\leq C r^{2m-\alpha c_m}$ where $C$ does not depend on $n\geq n_0$  and
$r\in(0,d(y_0)/9)$.

The conclusion now follows. Indeed, since
\(2m-\alpha c_m>0\) and the constants in the preceding estimates
are independent of \(n\) and \(r\), we may fix
\(0<\widetilde r<d(y_0)/9\) sufficiently small and choose
\(N_1\in\mathds{N}\) such that
\[
J_{2,n}(\widetilde r)+J_{3,n}(\widetilde r)
<
\frac{\varepsilon}{2}
\quad
\forall\,n\geq N_1.
\] Next, let $N_2 \in \mathds{N}$ such that $n \geq N_2$ implies $J_{1,n}(\tilde{r})<\varepsilon/2.$ Thus, considering $N=\max\{N_1,N_2\},$ we see that $n \geq N$ yields $|J_{\Omega,\alpha}(y_n)-J_{\Omega,\alpha}(y_0)|<\varepsilon$, concluding the proof.
\end{proof}
\begin{lemma}\label{Jalpha at the boundary}
Let \(\alpha\in(0,\Lambda_m)\). If $(y_n)\subseteq \Omega$ and \(y_n\to y_0\in\partial\Omega\) as $n \to + \infty$, then
\[
    \lim_{n \to + \infty} J_{\Omega,\alpha}(y_n)= |\Omega|.
\]
\end{lemma}
\begin{proof} Let $y_0 \in \partial \Omega$ and $(y_n)\subseteq \Omega$ be a sequence such that $y_n \to y_0$. Write
$$
    J_{\Omega,\alpha}(y_n)=\int_{B(y_n,d(y_n)/4)} e^{\alpha|G_{\Omega}(x,y_n)|}\,\mathrm{d}x+\int_{\Omega \setminus B(y_n,d(y_n)/4)} e^{\alpha|G_{\Omega}(x,y_n)|}\,\mathrm{d}x
    :=I_n+L_n.
$$
We will prove that $I_n \to 0$ and $L_n \to |\Omega|$ as $n \to + \infty$.

Let us first analyze $I_n.$ By \eqref{int exp near pole} we obtain
\begin{equation*}
    I_n\leq
C\frac{d(y_n)^{2m}}{\Lambda_m-\alpha}\longrightarrow0\quad \text{as}\quad n \to+\infty,
\end{equation*}
indeed $d(y_n)\to 0$ as $n \to + \infty$, since $y_n\to y_0 \in \partial \Omega$ as $n \to + \infty$.

Now, in order to study $L_n,$ define
\(
    E_n \coloneqq \Omega \setminus B(y_n,d(y_n)/4)\).
Thus
\begin{equation*}
    L_n = \int_{\Omega} \chi_{E_n}(x)e^{\alpha |G_{\Omega}(x,y_n)|}\,\mathrm{d}x.
\end{equation*}
Next, let $x \in \Omega$ be fixed and observe that since $y_n \to y_0 \in \partial \Omega,$ we have $x \in E_n$ for all $n \in \mathds{N}$ sufficiently large, showing that $\chi_{E_n}(x) \to 1$ for every $x \in \Omega.$ Moreover, for each $x \in \Omega$, by the continuity of the Green function $G_{\Omega}$ for pairs of points away from the diagonal we see that $G_{\Omega}(x,y_n)\to G_{\Omega}(x,y_0)=0$, since $y_0 \in \partial \Omega.$ In this way
\(
    \chi_{E_n}(x)e^{\alpha |G_{\Omega}(x,y_n)|}\to 1\) in \(\Omega\).
Recall from \eqref{int exp faraway pole} that 
\begin{equation*}
    \chi_{E_n}(x)e^{\alpha |G_{\Omega}(x,y_n)|}\leq e^{\alpha C} \in L^1(\Omega).
\end{equation*}
Thus $L_n \to |\Omega|$ by the Lebesgue dominated convergence theorem.
\end{proof}
\begin{proposition}\label{Jalpha sup attained}
For every \(\alpha\in(0,\Lambda_m)\), there exists \(y_\alpha\in\Omega\) such that
\[
    J_{\Omega,\alpha}(y_\alpha)=\max_{y\in \Omega}J_{\Omega,\alpha}(y)>|\Omega|.
\]

Consequently, the set $K_{\Omega,\alpha}$ defined in
\eqref{intro-maximizer-set} is nonempty and compact.
\end{proposition}
\begin{proof} For each $\alpha \in (0,\Lambda_m)$, denote
\begin{equation*}
    M_{\alpha}\coloneq \sup_{y \in \Omega} J_{\Omega,\alpha}(y).
\end{equation*}
From \hyperref[exp uniform G integrability]{Theorem \ref{exp uniform G integrability}} we already know that $M_{\alpha}<+\infty.$ We first observe that $M_{\alpha}>|\Omega|$. Indeed, let $y \in \Omega$ be fixed and apply \eqref{eq:GreenFunctLarge} to find a small radius $\tilde r>0$ such that
\begin{equation*}
    G_{\Omega}(x,y)\geq 1\quad\forall\, x \in B(y,\tilde{r})\setminus \{y\}. 
\end{equation*}
In this way,
$$
    J_{\Omega,\alpha}(y)=\int_{B(y,\tilde{r})} e^{\alpha |G_{\Omega}(x,y)|}\,\mathrm{d}x+\int_{\Omega\setminus B(y,\tilde{r})} e^{\alpha |G_{\Omega }(x,y)|}\,\mathrm{d}x
    \geq e^{\alpha}|B(y,\tilde{r})|+|\Omega\setminus B(y,\tilde{r})|
    >|\Omega|,
$$
proving that $M_{\alpha}>|\Omega|.$

Next, consider a sequence $(y_n)\subseteq \Omega$ such that $J_{\Omega,\alpha}(y_n) \to M_{\alpha}.$ By the compactness of \(\overline\Omega\), after passing to a subsequence, we may assume that $y_n \to y_{\alpha} \in \overline{\Omega}.$ Suppose, for the sake of contradiction, that $y_{\alpha} \in \partial \Omega.$ Thus, by \hyperref[Jalpha at the boundary]{Lemma \ref{Jalpha at the boundary}} we have $J_{\Omega,\alpha}(y_n) \to |\Omega|$ and, by the uniqueness of limits, $|\Omega|=M_{\alpha},$ a contradiction. This proves that $y_{\alpha} \in \Omega$ and by \hyperref[Jalpha continuous]{Lemma \ref{Jalpha continuous}} we conclude that $J_{\Omega,\alpha}(y_n) \to J_{\Omega,\alpha}(y_{\alpha})$ and again, by the uniqueness of limits, we have $J_{\Omega,\alpha}(y_{\alpha})=M_{\alpha}$ and the proof is complete.

It remains to prove that $K_{\Omega,\alpha}$ is compact. Let
$(x_n)\subseteq K_{\Omega,\alpha}$. Again, by the compactness of \(\overline\Omega\), after passing to a subsequence, we may assume that
\(
x_n\longrightarrow x\in\overline{\Omega}\).
If $x\in\partial\Omega$, then \hyperref[Jalpha at the boundary]{Lemma \ref{Jalpha at the boundary}}
gives
\[
\lim_{n \to + \infty }J_{\Omega,\alpha}(x_n)=|\Omega|,
\]
which contradicts
\[
J_{\Omega,\alpha}(x_n)=\max_{y\in\Omega}J_{\Omega,\alpha}(y)>|\Omega|.
\]
Therefore, $x\in\Omega$. By \hyperref[Jalpha continuous]{Lemma \ref{Jalpha continuous}},
\[
J_{\Omega,\alpha}(x)
=
\lim_{n\to\infty}J_{\Omega,\alpha}(x_n)
=
\max_{y\in\Omega}J_{\Omega,\alpha}(y),
\]
and hence $x\in K_{\Omega,\alpha}$. Thus, $K_{\Omega,\alpha}$ is compact.
\end{proof}
\section{Existence and characterization of extremals for $S_{\Omega,\mathcal P}(\alpha)$} \label{section6}

In this section, we prove the characterization
of extremal measures stated in
\hyperref[extremal characterization]{Theorem \ref{extremal characterization}}. We first
consider general smooth bounded domains and then specialize the
analysis to the unit ball.

Let $\alpha \in (0, \Lambda_m)$. \hyperref[Jalpha sup attained]{Proposition \ref{Jalpha sup attained}} shows that
$J_{\Omega,\alpha}$ attains its maximum in $\Omega$, and ensures that the set
$K_{\Omega,\alpha}$ defined in \eqref{intro-maximizer-set} is nonempty and compact. It
remains to identify all
probability measures that attain $S_{\Omega,\mathcal P}(\alpha)$.

\begin{proof}[Proof of Theorem~\ref{extremal characterization}]
Set
\[
M_\alpha:=\max_{y\in\Omega}J_{\Omega,\alpha}(y).
\]
   By \hyperref[prop:equality-of-extremal-values]{Proposition \ref{prop:equality-of-extremal-values}} and \hyperref[Jalpha sup attained]{Proposition \ref{Jalpha sup attained}}, we already know that $S_{\Omega,\mathcal P}(\alpha)=M_{\alpha}.$ Moreover, applying the Jensen reduction
\eqref{Jensen upper bound}, with $\mu \in \mathcal P(\Omega)$ in place of $\mu_f$, we
obtain
\begin{equation}\label{extremal-chain}
\int_\Omega e^{\alpha|(G_\Omega\star\mu)(x)|}\,\mathrm{d}x
\leq
\int_\Omega J_{\Omega,\alpha}(y)\,\mathrm{d}\mu(y)
\leq M_\alpha.
\end{equation}
Now, let \(y_{\alpha}\in K_{\Omega,\alpha}\). Then
\(
(G_{\Omega}\star\delta_{y_{\alpha}})(x)
=
G_{\Omega}(x,y_{\alpha})\) for a.e. 
\(x\in\Omega\).
Consequently,
\[
\int_{\Omega}
e^{\alpha|(G_{\Omega}\star\delta_{y_{\alpha}})(x)|}
\,\mathrm{d}x
=
J_{\Omega,\alpha}(y_{\alpha})
=
M_{\alpha}
=
S_{\Omega,\mathcal P}(\alpha).
\]
Thus, every Dirac mass \(\delta_y\) with \(y\in K_{\Omega,\alpha}\) is
a maximizer for $S_{\Omega,\mathcal P}(\alpha)$.

Finally, we prove that the only extremal measures for $S_{\Omega,\mathcal P}(\alpha)$ are of the form $\delta_{y}$ with $y \in K_{\Omega,\alpha}$. Indeed, let $\mu \in \mathcal{P}(\Omega)$ be such that
\begin{equation*}
    \int_{\Omega} e^{\alpha|G_{\Omega}\star \mu|}\,\mathrm{d}x=M_{\alpha}.
\end{equation*}
Thus, from \eqref{extremal-chain} we see that
\begin{equation*}
    \int_{\Omega} \Bigl(M_{\alpha}-J_{\Omega,\alpha}(y)\Bigr)\,\mathrm{d}\mu(y)=0,
\end{equation*}
and since $M_{\alpha}-J_{\Omega,\alpha}(y)\geq 0$ in $\Omega$, we see that \(M_\alpha-J_{\Omega,\alpha}(y)=0\) for \(\mu\)-a.e. \(y\in\Omega\), that is, $\mu(\Omega \setminus K_{\Omega,\alpha})=0$, showing that $\mu(K_{\Omega,\alpha})=1$, because $\mu(\Omega)=1$.

We claim that there exists $y_0 \in K_{\Omega,\alpha}$ such that $B(y_0,r)\subset\subset \Omega$ and $\mu(B(y_0,r))>0$ for every $r>0$ sufficiently small. Indeed, if this is not the case, it follows that for any $y \in K_{\Omega,\alpha},$ there exists $r_y>0$ such that $B(y,r_y)\subset \subset \Omega$ and $\mu(B(y,r_y))=0.$ Thus, $\{B(y,r_y)\}_{y \in K_{\Omega,\alpha}}$ is an open cover of $K_{\Omega,\alpha}$ and, by the compactness of $K_{\Omega,\alpha},$ we can extract a finite open subcover $\{B(y_i,r_{y_i})\}_{i \in \{1,\dots,N\}}$ of $K_{\Omega,\alpha}.$ In this way,
\begin{equation*}
    \mu(K_{\Omega,\alpha})\leq \sum_{i=1}^{N} \mu(B(y_i,r_{y_i}))=0,
\end{equation*}
which contradicts the fact that $\mu(K_{\Omega,\alpha})=1$.

Next, we claim that there is only one point with this property. Indeed, suppose for the sake of contradiction that there exists another point $y_1 \in K_{\Omega,\alpha}$ such that $y_1 \neq y_0,$ $B(y_1,r)\subset \subset \Omega$ and $\mu(B(y_1,r))>0$ for every $r>0$ sufficiently small. Since \(\mu\) is extremal, equality holds at every step of
\eqref{Jensen upper bound} with $\mu$ in place of $\mu_f$. In particular, equality holds in Jensen's
inequality for a.e. \(x\in\Omega\). Therefore, there exists
a measurable set \(X\subseteq\Omega\), with
\(
|\Omega\setminus X|=0\),
such that, for every \(x\in X\),
\[
\exp\left(
\alpha\int_{\Omega}|G_{\Omega}(x,y)|\,\mathrm{d}\mu(y)
\right)
=
\int_{\Omega}
e^{\alpha|G_{\Omega}(x,y)|}\,\mathrm{d}\mu(y).
\]
For each \(x\in X\), we apply the equality case in Jensen's
inequality to the probability space \((\Omega,\mu)\), the strictly
convex function \(t\mapsto e^t\), and the function
\(
y\mapsto\alpha|G_{\Omega}(x,y)|\).
By \cite[Theorem~2.2]{Lieb}, there exists a constant
\(c(x)\geq0\) such that
\[
|G_{\Omega}(x,y)|=c(x)
\qquad
\text{for }\mu\text{-a.e. }y\in\Omega.
\] Fix \(x_0\in X\setminus\{y_0,y_1\}\). We observe that
\[
|G_{\Omega}(x_0,y_0)|
=
|G_{\Omega}(x_0,y_1)|
=
c(x_0).
\]
Indeed, if $|G_{\Omega}(x_0,y_0)|\neq c(x_0)$, we can find $\rho>0$ small enough such that $|G_{\Omega}(x_0,y)|\neq c(x_0)$ for every $y \in B(y_0,\rho)$, because $y_0 \neq x_0$ and $y \mapsto G_{\Omega}(x_0,y)$ is continuous when $y$ is far away from $x_0.$ However, $\mu(B(y_0,\rho))>0,$ which contradicts the fact that \(|G_\Omega(x_0,y)|=c(x_0)\) for \(\mu\)-a.e. \(y\in\Omega\). By the same argument we also have $|G_{\Omega}(x_0,y_1)|=c(x_0).$ Since \(x_0\in X\setminus\{y_0,y_1\}\) was arbitrary, we conclude that
\begin{equation}\label{eq:two-poles-equality-measures}
    |G_\Omega(x,y_0)|=|G_\Omega(x,y_1)|
\qquad\text{for a.e. }x\in\Omega.
\end{equation}
Since $y_0\neq y_1$, by
\eqref{Green pole separation} there exists $\rho>0$ such that
\[
|G_\Omega(x,y_1)|>|G_\Omega(x,y_0)|
\quad
\forall\,
x\in B(y_1,\rho)\setminus\{y_1\}.
\]
Since $|B(y_1,\rho)|>0$, this contradicts
\eqref{eq:two-poles-equality-measures}. Hence, $y_0$ is unique.

The final step is to prove $\mu= \delta_{y_0}$. Since $y_0 \in K_{\Omega,\alpha}$ is the unique point such that $\mu(B(y_0,r))>0$ for every $r>0$ small enough satisfying $B(y_0,r)\subset\subset \Omega,$ for every \(z\in K_{\Omega,\alpha}\setminus\{y_0\}\), we can find a radius \(r_z>0\) such that $\mu(B(z,r_z))=0.$ Thus, for every sufficiently small \(\varepsilon>0\) the set $K_{\Omega,\alpha}\setminus B(y_0,\varepsilon)$ is compact and is covered by $\{B(z,r_z)\}_{z \in K_{\Omega,\alpha}\setminus B(y_0,\varepsilon)}$. We may therefore extract a finite subcover and conclude that \(\mu(K_{\Omega,\alpha}\setminus B(y_0,\varepsilon))=0\). Since \(\mu(K_{\Omega,\alpha})=1\), it follows that \(\mu(K_{\Omega,\alpha}\cap B(y_0,\varepsilon))=1\), and hence \(\mu(B(y_0,\varepsilon))=1\). Now, we may choose
\(n_0\in\mathds{N}\) such that
\(
\mu(B(y_0,1/n))=1
\) for every \(n\geq n_0\).
By the continuity of the measure from above,
\[
\mu(\{y_0\})
=
\mu\left(
\bigcap_{n=n_0}^{\infty}B(y_0,1/n)
\right)
=
\lim_{n\to+\infty}\mu(B(y_0,1/n))
=
1.
\]
This is enough to conclude the result. Indeed, if $A\subseteq \Omega$ is a Borel subset of $\Omega$ such that $y_0 \in A$ then
$1=\mu(\{y_0\})\leq \mu(A)\leq 1$, that is, $\mu(A)=1$. If $y_0 \not \in A$, then $y_0 \in \Omega \setminus A$ and $\mu(\Omega \setminus A)=1$, that is, $\mu(A)=0.$ This is precisely the definition of $\delta_{y_0}$. The conclusion then follows.
\end{proof}

Next, we specialize the analysis to the unit ball. Boggio's explicit
formula allows us to prove directly that $J_{B,\alpha}(y)$ depends only
on $|y|$ and is strictly decreasing as $|y|$ increases. This proves
\hyperref[thm:intro-ball]{Theorem \ref{thm:intro-ball}}.
\begin{proof}[Proof of Theorem~\ref{thm:intro-ball}]
By the positivity of Boggio's Green function, we have
\[
|G_B(x,y)|=G_B(x,y)
\quad
\forall\,x,y\in B,\quad x\neq y.
\]
Thus, the absolute value may be omitted throughout the proof. Specializing \eqref{eq:Boggio-explicit} to $x_0=0$ and $R=1$,
the quantity $[XY]$ satisfies
\[
    [XY]^2
    =
    1-2\langle x,y\rangle+|x|^2|y|^2.
\]
A direct calculation shows that, for every $x\neq y$ in $B$,
\begin{equation}\label{rhoy justification}
    [XY]^2-|x-y|^2
    =
    (1-|x|^2)(1-|y|^2)>0.
\end{equation}
Thus, defining
\[
    \rho_y(x):=\frac{|x-y|}{[XY]},
\]
we have from \eqref{rhoy justification} that $0<\rho_y(x)<1$ for every $x \neq y$ in $B$. We also define
\[
    \Psi(s):=
    c_m\left[
    \log\left(\frac{1}{s}\right)
    +
    \sum_{k=1}^{m-1}
    \frac{(-1)^k}{2k}\binom{m-1}{k}
    (1-s^{2k})
    \right],
    \qquad 0<s<1.
\]
Then \eqref{eq:Boggio-explicit} gives
\[
    G_B(x,y)=\Psi(\rho_y(x)).
\]
It follows directly that \(\Psi\) is continuous in $(0,1)$, $\lim_{s \to 0^{+}}\Psi(s)=+\infty$  and  $\lim_{s\to 1^{-}}\Psi(s)=0$. Moreover, $\Psi$ is strictly decreasing in $(0,1).$ Indeed,
\begin{align*}
    \Psi'(s)&=c_m\left[-\frac{1}{s}-\sum_{k=1}^{m-1}(-1)^{k}\binom{m-1}{k}s^{2k-1}\right]
    =-\frac{c_m}{s}\left[\sum_{k=0}^{m-1}(-1)^{k}\binom{m-1}{k}s^{2k}\right]\\
    &=-\frac{c_m}{s}(1-s^{2})^{m-1}
    <0.
\end{align*}
Now, we define $H_{\alpha}:(0,1)\to (1,+\infty)$ by
\(
    H_{\alpha}(s)\coloneqq e^{\alpha \Psi(s)}\).
Then, from the properties of $\Psi$, we see that $H_{\alpha}$ is a continuous, strictly decreasing function, $\lim_{s\to0^{+}} H_{\alpha}(s)=+\infty$ and $\lim_{s \to 1^{-}}H_{\alpha}(s)=1.$ Thus, $H_{\alpha}$ is a strictly decreasing bijection.

We now prove that $J_{B,\alpha}$ is radial. We begin by calculating the measure of the level sets
\begin{equation*}
    E_{y,t}\coloneqq \{x \in B: \rho_y(x)<t\},\quad y \in B\quad \text{and}\quad t\in (0,1).
\end{equation*}
Observe that $\rho_y(x)<t$ is equivalent to $|x-y|^2<t^2[XY]^2$, which is equivalent to
\begin{equation*}
    |x|^2-2\langle x,y\rangle +|y|^2<t^2(1-2\langle x,y\rangle+|y|^2|x|^2),
\end{equation*}
that is,
\begin{equation}\label{level set}
    (1-t^2|y|^2)|x|^2-2(1-t^2)\langle x,y \rangle+(|y|^2-t^2)<0,
\end{equation}
and since $0<t<1$ and $0\leq |y| <1$, we have $1-t^2|y|^2>0$ and a direct computation shows that \eqref{level set} is equivalent to
\begin{equation*}
    \left|x-\frac{1-t^2}{1-t^2|y|^2}y\right|^2<\left(\frac{t(1-|y|^2)}{1-t^2|y|^2}\right)^2.
\end{equation*}
We also observe that the ball determined by the preceding inequality
is contained in \(B\). Indeed, we have,
$$
\left|
\frac{1-t^2}{1-t^2|y|^2}y
\right|
+
\frac{t(1-|y|^2)}{1-t^2|y|^2}
=
\frac{(1-t^2)|y|+t(1-|y|^2)}
     {1-t^2|y|^2}
=
1-
\frac{(1-|y|)(1-t)(1-|y|t)}
     {1-t^2|y|^2}
<1.
$$
Therefore, the ball is compactly contained in \(B\), and hence
\[
E_{y,t}
=
B\left(
\frac{1-t^2}{1-t^2|y|^2}y,
\frac{t(1-|y|^2)}{1-t^2|y|^2}
\right),
\]
showing that,
\begin{equation*}
    |E_{y,t}|=|B|\left(\frac{t(1-|y|^2)}{1-t^2|y|^2}\right)^{2m}.
\end{equation*}
Next, observe that
\begin{equation*}
    J_{B,\alpha}(y)=\int_{B} e^{\alpha G_B(x,y)}\,\mathrm{d}x=\int_{B} e^{\alpha\Psi(\rho_y(x))}\,\mathrm{d}x=\int_{B} H_{\alpha}(\rho_y(x))\,\mathrm{d}x.
\end{equation*}
Thus, by the layer-cake representation
\begin{equation*}
    J_{B,\alpha}(y)=|B|+\int_{1}^{\infty} |\{x \in B: H_{\alpha}(\rho_y(x))>\tau\}|\,\mathrm{d}\tau.
\end{equation*}
Recalling that $H_{\alpha}$ is bijective, we know that for each $\tau>1$, there exists $t_{\tau} \in (0,1)$ such that $H_{\alpha}(t_{\tau})=\tau$, and since $H_{\alpha}$ is also strictly decreasing, we see that $H_{\alpha}(\rho_y(x))>\tau$ if and only if $\rho_y(x)<t_\tau.$ These two observations yield
\(
    \{x \in B: H_{\alpha}(\rho_y(x))>\tau\}=E_{y,t_{\tau}}\).
Thus,
\begin{equation*}
    J_{B,\alpha}(y)=|B|+\int_{1}^{\infty} |E_{y,t_{\tau}}|\,\mathrm{d}\tau,
\end{equation*}
and since $|E_{y,t_{\tau}}|$ depends only on $|y|$, the same is true of $J_{B,\alpha}(y).$ This proves that there exists a function $F_{\alpha}:[0,1)\to\mathds{R}$ such that $J_{B,\alpha}(y)=F_{\alpha}(|y|)$ for every $y \in B.$

Fix $t \in (0,1)$. The radius of the sets $E_{y,t}$ when $|y|=r$ is given by
\begin{equation*}
    R_{r,t}=\frac{t(1-r^2)}{1-t^2r^2}.
\end{equation*}
We claim that the map $r \mapsto R_{r,t}$ is strictly decreasing in $(0,1).$ Indeed,
\begin{align*}
    \frac{\mathrm{d}}{\mathrm{d}r} \log R_{r,t}&=\frac{\mathrm{d}}{\mathrm{d}r}[\log t+\log(1-r^2)-\log(1-t^2r^2)]\\
    &=-\frac{2r}{1-r^2}+\frac{2rt^2}{1-t^2r^2}
    =-\frac{2r(1-t^2)}{(1-r^2)(1-t^2r^2)}
    <0.
\end{align*}
Hence, by the strict monotonicity of the function $\log$, the map $r\mapsto R_{r,t}$ is strictly decreasing on \([0,1)\). This proves the claim. Thus, if $0\leq r_1<r_2<1$ we have $R_{r_1,t}>R_{r_2,t}$ for every $t \in (0,1)$, implying that $|E_{r_{1}v,t}|>|E_{r_{2}v,t}|$ for every $t \in (0,1)$ and every unit vector $v.$ In particular, for every $\tau>1$ we also have $|E_{r_{1}v,t_{\tau}}|>|E_{r_{2}v,t_{\tau}}|$. In this way,
\begin{equation*}
    F_{\alpha}(r_1)=|B|+\int_{1}^{\infty}|E_{r_{1}v,t_{\tau}}|\,\mathrm{d}\tau>|B|+\int_{1}^{\infty}|E_{r_{2}v,t_{\tau}}|\,\mathrm{d}\tau=F_{\alpha}(r_2).
\end{equation*}

Finally, since
\(
J_{B,\alpha}(y)=F_{\alpha}(|y|)\)
for every
\(y\in B
\)
and \(F_{\alpha}\) is strictly decreasing, we have
\[
J_{B,\alpha}(y)<J_{B,\alpha}(0)
\quad
\forall\,y\in B\setminus\{0\}.
\]
Therefore,
\(
K_{B,\alpha}=\{0\}\),
and
\[
\max_{y\in B}J_{B,\alpha}(y)=J_{B,\alpha}(0).
\]

\hyperref[extremal characterization]{Theorem
\ref{extremal characterization}} then implies that $\delta_0$ is the
unique extremal measure in $\mathcal P(B)$.
\end{proof}
The preceding theorem reduces the computation of the optimal value in the ball to the evaluation of $J_{B,\alpha}$ at the origin. We now use this fact to prove \hyperref[sharp BM ball]{Corollary \ref{sharp BM ball}}.

\begin{proof}[Proof of Corollary~\ref{sharp BM ball}]
By \hyperref[prop:L1-value-and-nonattainment]{Proposition \ref{prop:L1-value-and-nonattainment}} and
\hyperref[thm:intro-ball]{Theorem \ref{thm:intro-ball}}, the optimal
value in the $L^1$-class is
\[
S_{B,L^1}(\alpha)=J_{B,\alpha}(0).
\]
Applying \eqref{Green especial} with $x_0=0$ and $R=1$,
and then using polar coordinates, we obtain
\begin{align}
J_{B,\alpha}(0)
&=
\int_B e^{\alpha G_B(x,0)}\,\mathrm{d}x \nonumber\\
&=
|\mathds S^{2m-1}|
\int_0^1
r^{2m-1-\alpha c_m}
\exp\left[
\alpha c_m
\sum_{k=1}^{m-1}
\frac{(-1)^k}{2k}
\binom{m-1}{k}
\left(1-r^{2k}\right)
\right]\,\mathrm{d}r. \label{eq:Fontana-Morpurgo}
\end{align}
This proves the stated inequality and shows that its right-hand side
is optimal. The sharpness of the range
$\alpha\in(0,\Lambda_m)$ follows from
\hyperref[main sharp BM]{Theorem \ref{main sharp BM}}.
\end{proof}

\begin{remark} \label{rmk:Fontana-Morpurgo}
When $m=1$, the sum in \eqref{eq:Fontana-Morpurgo} is zero (see \hyperref[case m=1 remark]{Remark \ref{case m=1 remark}}), $c_1=1/(2\pi)$, and therefore
\[
J_{B,\alpha}(0)
=
2\pi\int_0^1 r^{1-\alpha/(2\pi)}\,\mathrm{d}r
=
\frac{4\pi^2}{4\pi-\alpha}.
\]
Consequently, when $m=1$, the estimate of
\hyperref[sharp BM ball]{Corollary \ref{sharp BM ball}} coincides with
\eqref{eq:fontana-morpurgo-bound} for $\Omega=B$, thereby recovering
the sharp estimate of Fontana and Morpurgo
\cite[Corollary~2]{FontanaMorpurgo}.
\end{remark}

\section{Asymptotic concentration and the Robin function}\label{section4}
Having characterized the extremal measures for
$S_{\Omega,\mathcal P}(\alpha)$, with $\alpha\in(0,\Lambda_m)$, we finally
investigate the behavior of the optimal constant $S_\Omega(\alpha)$ and of the
maximizing set $K_{\Omega,\alpha}$ as $\alpha\uparrow\Lambda_m$. The analysis is based
on the asymptotic behavior of $J_{\Omega,\alpha}$ and on the Robin function
associated with $G_\Omega$. These ingredients will be used to determine
the critical growth of $S_\Omega(\alpha)$ and to study the limiting behavior
of the maximizing points $y_\alpha\in K_{\Omega,\alpha}$, which support the
extremal Dirac measures. 

\subsection{The Robin function and asymptotics of $J_{\Omega,\alpha}$}

We begin by deriving the estimates needed to describe the
asymptotic behavior of $J_{\Omega,\alpha}$ uniformly when $y$ ranges over
compact subsets of $\Omega$. We then complement this interior
analysis with a boundary estimate for the maximizing points of
$J_{\Omega,\alpha}$.

Let $R_\Omega$ be the Robin function  defined in \eqref{eq:robin-function}. From the
decomposition
\begin{equation*}
    G_\Omega(x,y)
    =
    c_m\log\frac{D}{|x-y|}
    -
    H_{\Omega,y}(x), \qquad x,y\in\Omega,\quad x\neq y,
\end{equation*}
and the smoothness of $H_{\Omega,y}$ in a neighborhood of $y$, the limit
defining $R_\Omega$ is well defined and
\begin{equation}\label{Robin simpler def}
    R_\Omega(y)=c_m\log D-H_{\Omega,y}(y), \qquad \text{ for every }y\in\Omega.
\end{equation}
To analyze the critical behavior of $J_{\Omega,\alpha}$, we first establish a
local estimate for $G_\Omega(x,y)$ near the diagonal that is uniform
with respect to $y$ on compact subsets of $\Omega$. This is the
purpose of the following lemma.

\begin{lemma}\label{Robin technical lemma}
Let $\Omega\subset\mathds R^{2m}$ be a smooth bounded domain. Then, for every compact set $K\subset \subset\Omega$, there exist $r_K,C_K>0$ such that $B(y,r_{K})\subset \subset \Omega$ and

\begin{equation*}
    |H_{\Omega,y}(y)-H_{\Omega,y}(x)|\leq C_K|x-y| \quad \text{ for every } y\in K \text{ and } x \in B(y,r_K/2).
\end{equation*}

 In particular,
\begin{equation} \label{robin tech app 1}
\left|
G_\Omega(x,y)
-
c_m\log\frac1{|x-y|}
-
R_\Omega(y)
\right|
\leq C_{K}|x-y| \quad \text{ for every } y\in K \text{ and } x \in B(y,r_K/2)\setminus \{y\}.
\end{equation}

\end{lemma}

\begin{proof}Since $K \subset \subset \Omega$, there exists $d_{K}>0$ such that $d(y)\geq d_{K}$ for every $y \in K.$ Now, choose $0<r_{K}<d_{K}/4$ and observe that if $y \in K$ and $|x-y|<r_K$, we have 

$$d(x) \geq d(y)-|x-y|\geq d_K-r_K >\frac{3d_K}{4}.$$
Thus, \(B(y,r_K)\subset\subset\Omega\) for every \(y\in K\), with \(r_K\)
independent of \(y\).

Next, we apply the estimate \cite[(3.9)]{Mazya} due to Mayboroda and Maz'ya, with $n=2m$ and $i=j=0$, and obtain, 
\[
|H_{\Omega,y}(x)|
\leq
C\log\left(
1+\frac{D}
{\max\{d(x),d(y),|x-y|\}}
\right)
\quad
\forall\,x,y\in\Omega.
\]
Since \(d(y)\geq d_K\) for every \(y\in K\), we have
\(
\max\{d(x),d(y),|x-y|\}
\geq d(y)\geq d_K
\)
for every \(y\in K\) and \(x\in\Omega\). Therefore,
\[
|H_{\Omega,y}(x)|
\leq
C\log\left(1+\frac{D}{d_K}\right)
=:C_K
\qquad
\text{for every }y\in K
\text{ and }x\in\Omega.
\]
In particular,
\[
\|H_{\Omega,y}\|_{L^\infty(B(y,r_K))}
\leq C_K
\qquad
\text{for every }y\in K.
\]
In other words,
\begin{equation}\label{Robin property 1}
    \lVert H_{\Omega,y}\rVert_{L^{\infty}(B(y,r_K))}\leq C_{K}\quad \forall\,y \in K.
\end{equation}
Now, since $(-\Delta)^{m}H_{\Omega,y} = 0$ in $B(y,r_K)$, we can apply the interior estimate given on \cite[Theorem 2.20.]{Gazzola} with $p>2m$ to obtain
\begin{equation*}
    \lVert H_{\Omega,y} \rVert_{W^{2m,p}(B(y,r_K/2))}\leq C(r_K)\lVert H_{\Omega,y}\rVert_{L^1(B(y,r_{K}))},
\end{equation*}
and, from Morrey's inequality and \eqref{Robin property 1}, we get
\begin{equation*}
    \lVert \nabla H_{\Omega,y} \rVert_{L^{\infty}(B(y,r_K/2))}\leq C(r_K)\lVert H_{\Omega,y}\rVert_{L^1(B(y,r_{K}))}\leq C_K \quad \forall\,y \in K.
\end{equation*}
Finally, for each $x \in B(y,r_K/2)\setminus\{y\}$, with $y \in K$, consider the segment contained in $B(y,r_K/2)$ given by
\(
    \gamma(t)=y+t(x-y)\), \(0\leq t\leq 1\).
Then, from the fundamental theorem of calculus
\begin{equation*}
    |H_{\Omega,y}(x)-H_{\Omega,y}(y)|\leq \int_{0}^{1}|\nabla H_{\Omega,y}(y+t(x-y))||x-y|\,\mathrm{d}t\leq C_K|x-y|.
\end{equation*}

To complete the proof, observe that, for \(x\in B(y,r_K/2)\setminus\{y\}\), with $y \in K$, we have
$$
    G_\Omega(x,y)
-
c_m\log\frac1{|x-y|}
-
R_\Omega(y)
=H_{\Omega,y}(y)-H_{\Omega,y}(x).
$$
\end{proof}

The preceding estimate allows us to isolate, uniformly on compact
subsets of $\Omega$, the contribution of the logarithmic singularity
of $G_{\Omega}$ to $J_{\Omega,\alpha}$. We can now determine its precise
leading-order behavior as $\alpha\uparrow\Lambda_m$.

\begin{lemma}\label{Jalpha uniform critical behavior}
Let \(K\subset \subset\Omega\) be a compact set. Then
\[
    \lim_{\alpha\uparrow\Lambda_m}(\Lambda_m-\alpha)J_{\Omega,\alpha}(y)
    =
    2^{2m}\pi^{2m} e^{\Lambda_m R_\Omega(y)} \quad \text{ uniformly for } y\in K.
\]
\end{lemma}

\begin{proof}
Let $C_K,r_K>0$ be as in
\hyperref[Robin technical lemma]{Lemma~\ref{Robin technical lemma}},
so that $B(y,r_K)\subset\subset\Omega$ for every $y\in K$ and
\eqref{robin tech app 1} holds. We know from
\eqref{Robin property 1} in the proof of
\hyperref[Robin technical lemma]{Lemma~\ref{Robin technical lemma}}
that $|H_{\Omega,y}(y)|\leq C$ for every $y\in K$. Since
\(
R_\Omega(y)=c_m\log D-H_{\Omega,y}(y)
\) for every $y \in \Omega$,
we have
\[
M_K:=\sup_{y\in K}|R_\Omega(y)|<+\infty.
\]

Choose $\rho_0\in(0,r_K/2)$ so small that
\[
c_m\log\frac{1}{\rho_0}-M_K-C_K\rho_0>0.
\]
Let $\varepsilon>0$ be given. We then choose
$\rho_\varepsilon\in(0,\rho_0]$ so small that
\begin{equation}\label{rho epsilon choice}
\frac{|\mathds{S}^{2m-1}|}{c_m}
e^{\Lambda_mM_K}
\left(
e^{\Lambda_mC_K\rho_\varepsilon}-1
\right)
<
\frac{\varepsilon}{3}.
\end{equation}

Fix $y \in K$ and set
\[
A(y):=
e^{\Lambda_mR_\Omega(y)}
\frac{|\mathds{S}^{2m-1}|}{c_m}.
\]

First, we compare $A(y)$ with the two quantities
\[
e^{\Lambda_m(R_\Omega(y)\pm C_K\rho_\varepsilon)}
\frac{|\mathds{S}^{2m-1}|}{c_m}.
\]
For the positive sign, we have
\begin{equation*}
e^{\Lambda_m(R_\Omega(y)+C_K\rho_\varepsilon)}
\frac{|\mathds{S}^{2m-1}|}{c_m}
-A(y)
=
A(y)\left(
e^{\Lambda_mC_K\rho_\varepsilon}-1
\right)
\leq
\frac{|\mathds{S}^{2m-1}|}{c_m}
e^{\Lambda_mM_K}
\left(
e^{\Lambda_mC_K\rho_\varepsilon}-1
\right)
<
\frac{\varepsilon}{3},
\end{equation*}
where the last inequality follows from
\eqref{rho epsilon choice}. For the negative sign, we similarly obtain
\begin{align*}
A(y)-
e^{\Lambda_m(R_\Omega(y)-C_K\rho_\varepsilon)}
\frac{|\mathds{S}^{2m-1}|}{c_m}
&=
A(y)\left(
1-e^{-\Lambda_mC_K\rho_\varepsilon}
\right)\\
&\leq
A(y)\left(
e^{\Lambda_mC_K\rho_\varepsilon}-1
\right)
\leq
\frac{|\mathds{S}^{2m-1}|}{c_m}
e^{\Lambda_mM_K}
\left(
e^{\Lambda_mC_K\rho_\varepsilon}-1
\right)
<
\frac{\varepsilon}{3},
\end{align*}
where we also used the elementary inequality
\(
1-e^{-t}\leq e^t-1\),
\( t\geq0\).
Consequently, for every $y\in K$,
\begin{equation}\label{rho epsilon error}
A(y)-\frac{\varepsilon}{3}
<
e^{\Lambda_m(R_\Omega(y)-C_K\rho_\varepsilon)}
\frac{|\mathds{S}^{2m-1}|}{c_m}
\quad \text{ and } \quad
e^{\Lambda_m(R_\Omega(y)+C_K\rho_\varepsilon)}
\frac{|\mathds{S}^{2m-1}|}{c_m}
<
A(y)+\frac{\varepsilon}{3}.
\end{equation}

It follows from \eqref{robin tech app 1} that
\begin{equation}\label{G decompsition with Robin}
G_\Omega(x,y)
=
c_m\log\frac{1}{|x-y|}
+R_\Omega(y)+E(x,y) \quad \text{ for every } y\in K \text{ and } x\in B(y,r_K/2)\setminus\{y\},
\end{equation}
where
\[
|E(x,y)|\leq C_K|x-y|  \quad \text{ for every } y\in K \text{ and } x\in B(y,r_K/2)\setminus\{y\}.
\]
In particular,
\begin{equation}\label{Exy bounded}
|E(x,y)|\leq C_K\rho_\varepsilon
\end{equation}
for every $y\in K$ and 
$x\in B(y,\rho_\varepsilon)\setminus\{y\}$. Therefore,
$$
G_\Omega(x,y)
\geq
c_m\log\frac{1}{|x-y|}
-M_K-C_K\rho_\varepsilon\geq
c_m\log\frac{1}{\rho_0}
-M_K-C_K\rho_0
>0
$$
for every $y\in K$ and 
$x\in B(y,\rho_\varepsilon)\setminus\{y\}$. Consequently, $\lvert G_\Omega(x,y)\rvert=G_\Omega(x,y)$ throughout this region and, for every $y \in K$, we can decompose 
\[
\begin{aligned}
J_{\Omega,\alpha}(y)
&=
\int_{B(y,\rho_\varepsilon)}
e^{\alpha G_\Omega(x,y)}\,\mathrm{d}x
+
\int_{\Omega\setminus B(y,\rho_\varepsilon)}
e^{\alpha|G_\Omega(x,y)|}\,\mathrm{d}x\\
&=:
I_\alpha(y,\rho_\varepsilon)
+
L_\alpha(y,\rho_\varepsilon).
\end{aligned}
\]

We first estimate $L_\alpha(y,\rho_\varepsilon)$. By
\eqref{sharp Green function estimate}, if $y \in K$ and
$x\in\Omega\setminus B(y,\rho_\varepsilon)$, then
\[
|G_\Omega(x,y)|
\leq
C_\Omega
\log\left(
1+\frac{\min\{d(x),d(y)\}}{|x-y|}
\right)
\leq
C_\Omega\log\left(1+\frac{D}{\rho_\varepsilon}\right).
\]
Hence
\[
0\leq
(\Lambda_m-\alpha)L_\alpha(y,\rho_\varepsilon)
\leq
(\Lambda_m-\alpha)Q_\varepsilon \qquad \text{for every } y \in K \text{ and } \alpha \in (0, \Lambda_m),
\]
where
\(
Q_\varepsilon
:=
|\Omega|
\left(1+\frac{D}{\rho_\varepsilon}\right)^{\Lambda_mC_\Omega}
\). Define
\(
\alpha_\varepsilon^{L}
:=
\Lambda_m-
\min\left\{
\frac{\Lambda_m}{2},
\frac{\varepsilon}{6Q_\varepsilon}
\right\}
\).
Then $\alpha_\varepsilon^{L}\in(0,\Lambda_m)$ and
\begin{equation}\label{Lalpha epsilon estimate}
0\leq
(\Lambda_m-\alpha)L_\alpha(y,\rho_\varepsilon)
<
\frac{\varepsilon}{6} \qquad \text{for every } y \in K \text{ and } \alpha \in(\alpha_\varepsilon^{L},\Lambda_m).
\end{equation}

We next estimate $I_\alpha(y,\rho_\varepsilon)$. Let $y \in K$ and $\alpha \in (0, \Lambda_m)$. By
\eqref{G decompsition with Robin}, and \eqref{Exy bounded},
\begin{equation}\label{Ialpha squeeze 1}
e^{\alpha(R_\Omega(y)-C_K\rho_\varepsilon)}
\int_{B(y,\rho_\varepsilon)}
|x-y|^{-\alpha c_m}\,\mathrm{d}x
\leq I_\alpha(y,\rho_\varepsilon)
\leq
e^{\alpha(R_\Omega(y)+C_K\rho_\varepsilon)}
\int_{B(y,\rho_\varepsilon)}
|x-y|^{-\alpha c_m}\,\mathrm{d}x.
\end{equation}
Since $\Lambda_mc_m=2m$, we can compute
\[
\int_{B(y,\rho_\varepsilon)}
|x-y|^{-\alpha c_m}\,\mathrm{d}x
=
|\mathds{S}^{2m-1}|
\int_0^{\rho_\varepsilon}
r^{2m-1-\alpha c_m}\,\mathrm{d}r
=
\frac{|\mathds{S}^{2m-1}|}{c_m}
\frac{
\rho_\varepsilon^{c_m(\Lambda_m-\alpha)}
}{
\Lambda_m-\alpha 
}.
\]
Consequently,
\begin{equation}\label{Ialpha squeeze 2}
e^{\alpha(R_\Omega(y)-C_K\rho_\varepsilon)}
\rho_\varepsilon^{c_m(\Lambda_m-\alpha)}
\frac{|\mathds{S}^{2m-1}|}{c_m}
\leq
(\Lambda_m-\alpha)I_\alpha(y,\rho_\varepsilon) \qquad \text{for every } y \in K \text{ and } \alpha \in (0, \Lambda_m)
\end{equation}
and
\begin{equation}\label{Ialpha squeeze 3}
(\Lambda_m-\alpha)I_\alpha(y,\rho_\varepsilon)
\leq
e^{\alpha(R_\Omega(y)+C_K\rho_\varepsilon)}
\rho_\varepsilon^{c_m(\Lambda_m-\alpha)}
\frac{|\mathds{S}^{2m-1}|}{c_m} \qquad \text{for every } y \in K \text{ and } \alpha \in (0, \Lambda_m).
\end{equation}
Since $R_\Omega$ is bounded on $K$,
the lower and upper bounds in \eqref{Ialpha squeeze 2} and
\eqref{Ialpha squeeze 3} converge, as
$\alpha\uparrow\Lambda_m$, uniformly with respect to $y\in K$. More precisely,
$$\lim_{\alpha \uparrow \Lambda_m} e^{\alpha(R_\Omega(y)\pm C_K\rho_\varepsilon)}
\rho_\varepsilon^{c_m(\Lambda_m-\alpha)}
\frac{|\mathds{S}^{2m-1}|}{c_m}= e^{\Lambda_m(R_\Omega(y)\pm C_K\rho_\varepsilon)}
\frac{|\mathds{S}^{2m-1}|}{c_m} \qquad \text{ uniformly for } y \in K.$$
Therefore, there exists
$\alpha_\varepsilon^{I}\in(0,\Lambda_m)$ such that
\begin{equation}\label{Ialpha squeeze 4}
e^{\Lambda_m(R_\Omega(y)-C_K\rho_\varepsilon)}
\frac{|\mathds{S}^{2m-1}|}{c_m}
-\frac{\varepsilon}{6}
\leq
(\Lambda_m-\alpha)I_\alpha(y,\rho_\varepsilon)
\qquad \text{for every } y \in K \text{ and } \alpha \in (\alpha_\varepsilon^{I}, \Lambda_m)
\end{equation}
and
\begin{equation}\label{Ialpha squeeze 5}
(\Lambda_m-\alpha)I_\alpha(y,\rho_\varepsilon)
\leq
e^{\Lambda_m(R_\Omega(y)+C_K\rho_\varepsilon)}
\frac{|\mathds{S}^{2m-1}|}{c_m}
+\frac{\varepsilon}{6} \qquad \text{for every } y \in K \text{ and } \alpha \in (\alpha_\varepsilon^{I}, \Lambda_m).
\end{equation}

Finally, set
\(
\alpha_\varepsilon
:=
\max\{
\alpha_\varepsilon^{I},
\alpha_\varepsilon^{L}
\}\). Combining
\eqref{Ialpha squeeze 4} with the nonnegativity of
$L_\alpha(y,\rho_\varepsilon)$ gives
\begin{equation}\label{Ialpha squeeze 6}
e^{\Lambda_m(R_\Omega(y)-C_K\rho_\varepsilon)}
\frac{|\mathds{S}^{2m-1}|}{c_m}
-\frac{\varepsilon}{6}
\leq
(\Lambda_m-\alpha)J_{\Omega,\alpha}(y) \qquad \text{for every } y \in K \text{ and } \alpha \in (\alpha_\varepsilon, \Lambda_m).
\end{equation}
On the other hand, combining \eqref{Ialpha squeeze 5} with
\eqref{Lalpha epsilon estimate}, we obtain
\begin{equation}\label{Ialpha squeeze 7}
(\Lambda_m-\alpha)J_{\Omega,\alpha}(y)
\leq
e^{\Lambda_m(R_\Omega(y)+C_K\rho_\varepsilon)}
\frac{|\mathds{S}^{2m-1}|}{c_m}
+\frac{\varepsilon}{3} \qquad \text{for every } y \in K \text{ and } \alpha \in (\alpha_\varepsilon, \Lambda_m).
\end{equation}
It now follows from \eqref{rho epsilon error} and
\eqref{Ialpha squeeze 6} that
\[
A(y)-(\Lambda_m-\alpha)J_{\Omega,\alpha}(y)
<
\frac{\varepsilon}{3}+\frac{\varepsilon}{6}
<\varepsilon \qquad \text{for every } y \in K \text{ and } \alpha \in (\alpha_\varepsilon, \Lambda_m).
\]
Similarly, \eqref{rho epsilon error} and
\eqref{Ialpha squeeze 7} imply
\[
(\Lambda_m-\alpha)J_{\Omega,\alpha}(y)-A(y)
<
\frac{\varepsilon}{3}+\frac{\varepsilon}{3}
<\varepsilon \qquad \text{for every } y \in K \text{ and } \alpha \in (\alpha_\varepsilon, \Lambda_m).
\]
Therefore,
\[
\sup_{y\in K}
\left|
(\Lambda_m-\alpha)J_{\Omega,\alpha}(y)
-
e^{\Lambda_mR_\Omega(y)}
\frac{|\mathds{S}^{2m-1}|}{c_m}
\right|
<\varepsilon \qquad \text{for every } \alpha \in (\alpha_\varepsilon, \Lambda_m).
\]

This proves that
\[
\lim_{\alpha\uparrow\Lambda_m}
(\Lambda_m-\alpha)J_{\Omega,\alpha}(y)
=
e^{\Lambda_mR_\Omega(y)}
\frac{|\mathds{S}^{2m-1}|}{c_m}
\quad\text{uniformly for }y\in K.
\]
Finally,
\[
\frac{|\mathds{S}^{2m-1}|}{c_m}
=
\frac{\frac{2\pi^m}{(m-1)!}}
{\frac{1}{2^{2m-1}\pi^m(m-1)!}}
=
2^{2m}\pi^{2m},
\]
which completes the proof.
\end{proof}
\begin{remark}\label{R is continuos} The following observation is standard and is included only for completeness. As a by-product of the preceding lemma, $R_\Omega\in C(\Omega)$. Indeed, let
$K\subset\subset\Omega$ be compact. By the continuity of $J_{\Omega,\alpha}$
established in \hyperref[section5]{Section \ref{section5}}, the functions
\(
    y\mapsto(\Lambda_m-\alpha)J_{\Omega,\alpha}(y)
\)
are continuous on $K$. Since they converge as $\alpha \uparrow \Lambda_m$ uniformly on $K$ to
\(
    y\mapsto
    2^{2m}\pi^{2m}e^{\Lambda_mR_\Omega(y)},
\)
the latter function is continuous on $K$. Applying the logarithm, we
conclude that $R_\Omega$ is continuous on $K$. Since $K\subset\subset\Omega$ is
arbitrary, it follows that $R_\Omega\in C(\Omega)$.
\end{remark}

The uniform convergence established above is available only on compact
subsets of $\Omega$. Since the maximizing points of $J_{\Omega,\alpha}$ depend
on $\alpha$, before applying this convergence to them we must show that
they remain in a fixed compact subset of $\Omega$ as
$\alpha\uparrow\Lambda_m$. For this purpose, we combine the interior asymptotic estimate of \hyperref[Jalpha uniform critical behavior]{Lemma \ref{Jalpha uniform critical behavior}}
with the boundary estimate established in \eqref{Jalpha and distance to the boundary} to show that
the maximizing points cannot approach $\partial\Omega$ as
$\alpha\uparrow\Lambda_m$.

\begin{proposition}\label{Kalpha points boundary}
There exist $\delta>0$ and $\alpha_0\in(0,\Lambda_m)$ such that $d(y)\geq\delta$ for every $y\in K_{\Omega,\alpha}$ and every $\alpha\in(\alpha_0,\Lambda_m)$.
\end{proposition}
\begin{proof}Fix $y_0\in\Omega$. Since $\{y_0\}\subset \subset \Omega$, by \hyperref[Jalpha uniform critical behavior]{Lemma \ref{Jalpha uniform critical behavior}}, we have
\begin{equation*}
    \lim_{\alpha \uparrow \Lambda_m}(\Lambda_m-\alpha)J_{\Omega,\alpha}(y_0)=2^{2m}\pi^{2m}e^{\Lambda_mR_\Omega(y_0)}.
\end{equation*}
Observe that $2^{2m}\pi^{2m}e^{\Lambda_m R_\Omega(y_0)}>0$. Therefore, there exist $\alpha_1 \in (0,\Lambda_m)$ and $A>0$ such that 
\begin{equation*}
    (\Lambda_m-\alpha)J_{\Omega,\alpha}(y_0)\geq A \qquad \text{ for every } \alpha \in (\alpha_1,\Lambda_m).
\end{equation*}

Let $y_{\alpha}\in K_{\Omega,\alpha}$ with $\alpha \in (0,\Lambda_m)$. It follows that $(\Lambda_m-\alpha)J_{\Omega,\alpha}(y_{\alpha})\geq A$ for every $\alpha \in (\alpha_1,\Lambda_m).$ Now, by \eqref{Jalpha and distance to the boundary}, we also have
\begin{equation*}
    (\Lambda_m-\alpha)J_{\Omega,\alpha}(y_{\alpha})\leq C(\Lambda_m-\alpha)+Cd(y_{\alpha})^{2m} \qquad \text{ for every } \alpha \in (0,\Lambda_m).
\end{equation*}
Hence,
\begin{equation*}
    A\leq C(\Lambda_m-\alpha)+Cd(y_{\alpha})^{2m} \qquad \text{ for every } \alpha \in (\alpha_1,\Lambda_m).
\end{equation*}
Next, choose $\alpha_0 \in (\alpha_1,\Lambda_m)$ such that  $C(\Lambda_m-\alpha)\leq A/2$ for every $\alpha \in (\alpha_0,\Lambda_m)$. In this way
\begin{equation*}
     d(y_{\alpha})\geq \left(\frac{A}{2C}\right)^{\frac{1}{2m}}>0 \qquad \text{ for every } \alpha \in (\alpha_0,\Lambda_m).
\end{equation*}
Therefore, taking
\[
\delta
:=
\left(\frac{A}{2C}\right)^{1/(2m)},
\]
we obtain
\[
d(y_{\alpha})\geq\delta
\qquad
\text{ for every } \alpha\in(\alpha_0,\Lambda_m).
\]
Since $y_{\alpha}\in K_{\Omega,\alpha}$ was chosen arbitrarily and the
constants $\alpha_0$ and $\delta$ do not depend on this choice, the
estimate holds for every $y\in K_{\Omega,\alpha}$.
\end{proof}
\begin{remark}\label{Robin has max} For completeness, we record that the Robin function attains its
maximum in $\Omega$. Indeed, by \eqref{Green omega to ball}, after increasing $C$ if necessary,
\begin{equation*}
    G_{\Omega}(x,y)\leq C+c_m\log\frac{d(y)}{|x-y|}\qquad \text{ for every } y \in \Omega \text{ and $x \in B(y,d(y)/4)\setminus\{y\}$}.
\end{equation*}
Subtracting $c_m\log(1/|x-y|)$ from both sides and letting $x\to y$, we obtain
\begin{equation*}
    R_\Omega(y)\leq C+c_m\log d(y) \qquad \text{ for every } y \in \Omega.
\end{equation*}
Thus, $R_\Omega(y)\to-\infty$ as $d(y)\to0$. Now, fix $y_0\in\Omega$. Since $R_\Omega(y_0)<+\infty$, we may choose $\delta\in(0,d(y_0))$ such that $d(y)<\delta$ implies $R_\Omega(y)<R_\Omega(y_0)-1$. Consequently, $\sup_{\Omega}R_\Omega=\sup_{\{d(y)\geq\delta\}}R_\Omega$. The set $\{y\in\Omega:d(y)\geq\delta\}$ is compactly contained in $\Omega$, and $R_\Omega$ is continuous on this set by \hyperref[R is continuos]{Remark \ref{R is continuos}}. Therefore, the latter supremum is attained, and hence $R_\Omega$ attains its maximum in $\Omega$.
\end{remark}
\subsection{Asymptotic selection by the Robin function}

\hyperref[Jalpha uniform critical behavior]{Lemma \ref{Jalpha uniform critical behavior}} gives the
asymptotic behavior of $J_{\Omega,\alpha}$ as $\alpha \uparrow \Lambda_m$ uniformly on compact subsets of
$\Omega$, while
\hyperref[Kalpha points boundary]{Proposition \ref{Kalpha points boundary}} shows
that, for $\alpha$ sufficiently close to $\Lambda_m$, all the
maximizing points $y\in K_{\Omega,\alpha}$ belong to a fixed compact subset of
$\Omega$. Consequently, the asymptotic formula given in \hyperref[Jalpha uniform critical behavior]{Lemma \ref{Jalpha uniform critical behavior}} can be applied uniformly
at these maximizing points. This allows us to determine the
leading-order behavior of $S_\Omega(\alpha)$ and to identify the possible
limits of the points supporting extremal Dirac measures as
$\alpha\uparrow\Lambda_m$.

\begin{proof}[Proof of Theorem~\ref{intro-critical}] By \hyperref[Robin has max]{Remark \ref{Robin has max}}, we may choose
$y^{*}\in\Omega$ such that
\[
R_\Omega(y^{*})=\max_{y\in\Omega}R_\Omega(y).
\]
Moreover, by \hyperref[extremal characterization]{Theorem \ref{extremal characterization}},
\[
S_\Omega(\alpha)=\max_{y\in\Omega}J_{\Omega,\alpha}(y)\quad \forall\, \alpha \in (0,\Lambda_m).
\] Thus, for every $\alpha \in (0,\Lambda_m)$, we have $(\Lambda_m-\alpha)J_{\Omega,\alpha}(y^{\ast})\leq(\Lambda_m-\alpha) S_\Omega(\alpha).$ Now, given $\varepsilon>0$, by \hyperref[Jalpha uniform critical behavior]{Lemma \ref{Jalpha uniform critical behavior}} we can find $\alpha_1 \in (0,\Lambda_m)$ such that
\begin{equation*}
    \left |(\Lambda_m-\alpha)J_{\Omega,\alpha}(y^{\ast})-2^{2m}\pi^{2m}e^{\Lambda_mR_\Omega(y^{\ast})} \right|<\varepsilon, \quad \forall\, \alpha \in (\alpha_1,\Lambda_m).
\end{equation*}
Consequently,
\begin{equation}\label{Salpha asymptotic 1}
   2^{2m}\pi^{2m}e^{\Lambda_m R_\Omega(y^{\ast})}-\varepsilon<(\Lambda_m-\alpha)S_\Omega(\alpha)\quad \forall\,\alpha \in (\alpha_1,\Lambda_m). 
\end{equation}

Next, we estimate $(\Lambda_m-\alpha)S_\Omega(\alpha)$ from above. For each $\alpha \in (0,\Lambda_m)$, choose $y_{\alpha} \in K_{\Omega,\alpha}.$ By \hyperref[Kalpha points boundary]{Proposition \ref{Kalpha points boundary}}, there exist $\alpha_0\in (0,\Lambda_m)$ and $\delta>0$ such that $d(y_{\alpha})\geq\delta$ for every $\alpha \in (\alpha_0,\Lambda_m).$ Next, define the set $\Omega_{\delta}\coloneqq \{y \in \Omega:d(y) \geq \delta\}$ and observe that  $\Omega_{\delta}\subset\subset \Omega$ is compact and $y_{\alpha} \in \Omega_{\delta}$ for every $\alpha \in (\alpha_0,\Lambda_m).$ We now apply \hyperref[Jalpha uniform critical behavior]{Lemma \ref{Jalpha uniform critical behavior}} with  $K=\Omega_{\delta}$ and find $\alpha_2 \in (\alpha_0,\Lambda_m)$, such that
\begin{equation*}
    \left |(\Lambda_m-\alpha)J_{\Omega,\alpha}(y_{\alpha})-2^{2m}\pi^{2m}e^{\Lambda_mR_\Omega(y_{\alpha})} \right|<\varepsilon, \quad \forall\, \alpha \in (\alpha_2,\Lambda_m).
\end{equation*}
Recalling that $J_{\Omega,\alpha}(y_{\alpha})=S_\Omega({\alpha})$, we obtain
\begin{equation}\label{Salpha asymptotic 2}
 (\Lambda_m-\alpha) S_\Omega(\alpha)<2^{2m}\pi^{2m}e^{\Lambda_m\max_{\Omega}R_\Omega}+\varepsilon\quad \forall\, \alpha \in (\alpha_2,\Lambda_m).
\end{equation}
Choosing $\alpha_3\in
\left(\max\{\alpha_1,\alpha_2\},\Lambda_m\right)$ and recalling that $R_\Omega(y^{\ast})=\max_{\Omega} R_\Omega$, we have by \eqref{Salpha asymptotic 1} and \eqref{Salpha asymptotic 2}
\begin{equation*}
    2^{2m}\pi^{2m}e^{\Lambda_m \max_{\Omega}R_\Omega}-\varepsilon<(\Lambda_m-\alpha)S_\Omega(\alpha)<2^{2m}\pi^{2m}e^{\Lambda_m\max_{\Omega}R_\Omega}+\varepsilon\quad \forall \alpha \in (\alpha_3,\Lambda_m).
\end{equation*}
In other words,
\begin{equation*}
    \left |(\Lambda_m-\alpha)S_\Omega(\alpha)-2^{2m}\pi^{2m}e^{\Lambda_m\max_{\Omega}R_\Omega}\right|<\varepsilon,\quad \alpha \in (\alpha_3,\Lambda_m),
\end{equation*}
This proves the first assertion. We now turn to the asymptotic selection of the maximizing points.

Let $\mathcal{M}(\Omega)$ denote the space of finite Radon measures
on $\Omega$. We recall that
\(
\mu_k\rightharpoonup^{*}\mu
\) in \(\mathcal{M}(\Omega)\)
if and only if
\[
\lim_{k \to + \infty}\int_{\Omega}\varphi\,\mathrm{d}\mu_k
=
\int_{\Omega}\varphi\,\mathrm{d}\mu
\qquad
\text{ for every }\varphi\in C_0(\Omega),
\]
where
\(
C_0(\Omega)
:=
\left\{
\varphi\in C(\Omega):
\varphi \text{ admits a continuous extension }
\overline{\varphi}\in C(\overline{\Omega})
\text{ such that }
\overline{\varphi}=0 \text{ on }\partial\Omega
\right\}\). Equivalently, $C_0(\Omega)$ is the closure of the space $ C_c(\Omega)\coloneq\left\{
\varphi\in C(\Omega):
\mathrm{supp}\,\varphi\subset\subset\Omega
\right\}$ with respect to the $L^\infty(\Omega)$-norm.

Let $(\alpha_k)\subset(0,\Lambda_m)$ satisfy
$\alpha_k\uparrow\Lambda_m$ as $k \to + \infty$, and choose
$y_{\alpha_k}\in K_{\Omega,\alpha_k}$ for every $k\in\mathds{N}$.
 By \hyperref[extremal characterization]{Theorem \ref{extremal characterization}} we can write $S_\Omega(\alpha_k)=J_{\Omega,\alpha_k}(y_{\alpha_k})$ for every $k \in \mathds{N}$. 
As in the first part of the proof, by \hyperref[Kalpha points boundary]{Proposition \ref{Kalpha points boundary}}, there exist $\alpha_0 \in (0, \Lambda_m)$, $\delta>0$, and  $k_0\in\mathds N$ such that  $d(y_{\alpha})\geq\delta$ for every $\alpha \in (\alpha_0,\Lambda_m)$, $\alpha_k\in(\alpha_0,\Lambda_m)$ for every $k\geq k_0$,
and $y_{\alpha_k}\in \Omega_\delta \coloneqq \{y \in \Omega:d(y) \geq \delta\}$ for every $k\geq k_0$. Since
$\Omega_\delta\subset\subset\Omega$ is compact, up to a subsequence,
\(
    y_{\alpha_k}\rightarrow y^\ast\in \Omega_\delta\). 
Moreover, by the first assertion, we have  
    \begin{equation}\label{comparation 1}
        \lim_{k\to +\infty}(\Lambda_m-\alpha_k)J_{\Omega,\alpha_k}(y_{\alpha_k})=2^{2m}\pi^{2m}e^{\Lambda_m\max_{\Omega} R_\Omega}.
    \end{equation}
We can also apply \hyperref[Jalpha uniform critical behavior]{Lemma \ref{Jalpha uniform critical behavior}} on the compact set $\Omega_{\delta}\subset\subset \Omega$ and obtain in particular
\begin{equation}\label{comparation 2}
    \lim_{k \to+\infty}\left((\Lambda_m-\alpha_k)J_{\Omega,\alpha_k}(y_{\alpha_k})- 2^{2m}\pi^{2m}e^{\Lambda_m R_\Omega(y_{\alpha_k})} \right)= 0.
\end{equation}
Indeed, as $k \to +\infty$, we have 
\begin{equation*}
\left|(\Lambda_m-\alpha_k)J_{\Omega,\alpha_k}(y_{\alpha_k})- 2^{2m}\pi^{2m}e^{\Lambda_m R_\Omega(y_{\alpha_k})} \right|\leq \sup_{y \in \Omega_{\delta}}\left |(\Lambda_m-\alpha_k)J_{\Omega,\alpha_k}(y)- 2^{2m}\pi^{2m}e^{\Lambda_m R_\Omega(y)}\right|\to 0.
\end{equation*}
Thus, from \eqref{comparation 1} and \eqref{comparation 2}, we obtain
\(
    e^{\Lambda_m R_\Omega(y_{\alpha_k})}\to  e^{\Lambda_m \max_{\Omega} R_\Omega}\) as \(k\to+\infty\),
implying that $R_\Omega(y_{\alpha_k})\to \max_{\Omega} R_\Omega$. Now,
since $y_{\alpha_k} \to y^{\ast}$ and $R_\Omega$ is continuous on $\Omega_{\delta}$ (see \hyperref[R is continuos]{Remark \ref{R is continuos}}), we see that $R_\Omega(y_{\alpha_k}) \to R_\Omega(y^{\ast})$ and by the uniqueness of the limit, we conclude that $R_\Omega(y^{\ast})=\max_{\Omega} R_\Omega.$

Next, let $\varphi\in C_0(\Omega)$. Since
$y_{\alpha_k}\to y^{*}$, we have
\[
\lim_{k \to + \infty}\int_{\Omega}\varphi\,\mathrm{d}\delta_{y_{\alpha_k}}
=
\lim_{k \to + \infty}\varphi(y_{\alpha_k})
=
\varphi(y^{*})
=
\int_{\Omega}\varphi\,\mathrm{d}\delta_{y^{*}}.
\]
Therefore,
\(
\delta_{y_{\alpha_k}}
\rightharpoonup^{*}
\delta_{y^{*}}\)
 in \(\mathcal{M}(\Omega)\).

Finally, suppose that $R_\Omega$ has a unique maximum point $y^{*}$.
Suppose by contradiction that the whole sequence
$(y_{\alpha_k})$ does not converge to $y^{*}$. Then there exist
$\varepsilon_0>0$ and a subsequence $(y_{\alpha_{k_\ell}})$ such that
\begin{equation} \label{eq:ContrAsympt}
|y_{\alpha_{k_\ell}}-y^{*}|
\geq\varepsilon_0
\quad
\forall\,\ell\in\mathds{N}.
\end{equation}
Applying the first part of the proof to this subsequence, we can
extract a further subsequence converging to some
$\widetilde{y}\in\Omega$ satisfying
\(
R_\Omega(\widetilde{y})
=
\max_{y\in\Omega}R_\Omega(y)\). By the uniqueness of the maximum point, $\widetilde{y}=y^{*}$ and this contradicts \eqref{eq:ContrAsympt}. Therefore,
\(
y_{\alpha_k}\longrightarrow y^{*}
\)
for the whole sequence. Consequently,
\(
\delta_{y_{\alpha_k}}
\rightharpoonup^{*}
\delta_{y^{*}}
\) in \(\mathcal{M}(\Omega)\).
\end{proof}

\begin{remark} \label{Flucher}
When $m=1$, the first asymptotic formula in \hyperref[intro-critical]{Theorem \ref{intro-critical}}
reduces to
\[
\lim_{\alpha\uparrow4\pi}
(4\pi-\alpha)J_{\Omega,\alpha}(y)
=
4\pi^2e^{4\pi R_\Omega(y)}.
\]
For comparison, we recall that Flucher in \cite{Flucher} considered, for each
$y\in\Omega$, the concentration class
\[
C_y:=\left\{(u_k)\subset H_0^1(\Omega):
\|\nabla u_k\|_{L^2(\Omega)}\leq1\ \forall\,k,\quad
|\nabla u_k|^2\,\mathrm{d}x\rightharpoonup^\ast\delta_y\right\}.
\]
For the critical Trudinger--Moser functional, he proved that
\[
\sup_{(u_k)\in C_y}
\limsup_{k\to\infty}
\int_\Omega\bigl(e^{4\pi u_k^2}-1\bigr)\,\mathrm{d}x
=
\pi e\,r_\Omega(y)^2,
\]
see \cite[Lemma~4 and Theorem~13]{Flucher}.
Here the conformal radius $r_{\Omega}$ is defined by
\[
r_\Omega(y)
=
e^{-2\pi (H_{\Omega,y}(y)-c_m\log D)}
\]
where for each fixed $x \in \Omega$, the map $y\in \Omega\mapsto H_{\Omega,x}(y)$ is the regular part of the Green function $G_{\Omega}$ defined in \eqref{auxiliary}, see 
\cite[Section~4.2]{Flucher}. Comparing this decomposition
with \eqref{eq:robin-function} and \eqref{Robin simpler def}, we obtain
\[
r_\Omega(y)^2=e^{4\pi R_\Omega(y)}.
\]
Thus, up to universal multiplicative constants, both limiting
formulas depend on the concentration point through $e^{4\pi R_\Omega(y)}$;
hence the largest concentration levels correspond precisely to the
maximum points of $R_\Omega$.
\end{remark}

\subsection*{Statements and Declarations}

\noindent\textbf{Funding.}
This study was financed in part by the Coordenação de Aperfeiçoamento de Pessoal de Nível Superior - Brasil (CAPES) - Finance Code 001, and by the MATH-AmSud grant 88887.878894/2023-00. 
J. M. do \'O acknowledges partial support from CNPq
through grants 312435/2026-6, 409764/2023-0, and 443594/2023-6 and from the
Para\'iba State Research Foundation (FAPESQ), grant no. 3034/2021. 
F. Sani was partially supported by INdAM-GNAMPA.
V. R. Silva acknowledges partial support from CAPES through Grants 88887.824231/2023-00 and 88881.980555/2024-01.

\medskip

\noindent\textbf{Conflict of interest.}
The authors declare that they have no conflict of interest.

\medskip

\noindent\textbf{Authors' contributions.}
All authors contributed to the conception of the work, the development of the
mathematical arguments, and the writing and revision of the manuscript. All
authors read and approved the final manuscript.

\medskip

\noindent\textbf{Data availability.}
Data sharing does not apply to this article, as no datasets were generated or
analyzed during the current study.

\bigskip


\bibliographystyle{plainnat}
\bibliography{References}

@article{BM,
 author = {Brezis, Ha{\"{\i}}m and Merle, Frank},
 title = {Uniform estimates and blow-up behavior for solutions of {{\(-\Delta{} u ={V}(x) e^ u\)}} in two dimensions},
 fjournal = {Communications in Partial Differential Equations},
 journal = {Commun. Partial Differ. Equations},
 volume = {16},
 number = {8-9},
 pages = {1223--1253},
 year = {1991},
 language = {English},
 zbMATH = {28018},
 Zbl = {0746.35006}
}

@book{Gazzola,
 author = {Gazzola, Filippo and Grunau, Hans-Christoph and Sweers, Guido},
 title = {Polyharmonic boundary value problems. {Positivity} preserving and nonlinear higher order elliptic equations in bounded domains},
 fseries = {Lecture Notes in Mathematics},
 series = {Lect. Notes Math.},
 volume = {1991},
 year = {2010},
 publisher = {Berlin: Springer},
 language = {English},
 zbMATH = {5712793},
 Zbl = {1239.35002}
}

@article{Boggioclassical,
 author = {Boggio, T.},
 title = {Sulle funzioni di {G}reen d'ordine {{\(m\)}}.},
 fjournal = {Rendiconti del Circolo Matematico di Palermo},
 journal = {Rend. Circ. Mat. Palermo},
 volume = {20},
 pages = {97--135},
 year = {1905},
 language = {Italian},
 zbMATH = {2651738},
 JFM = {36.0827.01}
}

@Book{Lieb,
 Author = {Lieb, Elliott H. and Loss, Michael},
 Title = {Analysis},
 Edition = {2nd},
 FSeries = {Graduate Studies in Mathematics},
 Series = {Grad. Stud. Math.},
 Volume = {14},
 Year = {2001},
 Publisher = {Providence, RI: American Mathematical Society (AMS)},
 Language = {English},
 zbMATH = {1601796},
 Zbl = {0966.26002}
}

@article{ExistenceGreen,
 author = {Krasovskij, J. P.},
 title = {Isolation of singularities of the {Green}'s function},
 fjournal = {Mathematics of the USSR. Izvestiya},
 journal = {Math. USSR, Izv.},
 volume = {1},
 pages = {935--966},
 year = {1969},
 language = {English},
 zbMATH = {3277880},
 Zbl = {0174.15703}
}

@book{BrezisFA,
 author = {Brezis, Haim},
 title = {Functional analysis, {Sobolev} spaces and partial differential equations},
 fseries = {Universitext},
 series = {Universitext},
 year = {2011},
 publisher = {New York, NY: Springer},
 language = {English},
 zbMATH = {5633610},
 Zbl = {1220.46002}
}

@incollection{Mazya,
 author = {Mayboroda, Svitlana and Maz'ya, Vladimir},
 title = {Pointwise estimates for the polyharmonic {Green} function in general domains},
 booktitle = {Analysis, partial differential equations and applications. The Vladimir Maz'ya anniversary volume. Selected papers of the international workshop, Rome, Italy, June 30--July 3, 2008},
 pages = {143--158},
 year = {2009},
 publisher = {Basel: Birkh{\"a}user},
 language = {English},
 zbMATH = {5635475},
 Zbl = {1193.35032}
}

@book{FollandPDE,
 author = {Folland, Gerald B.},
 title = {Introduction to partial differential equations.},
 edition = {2nd},
 year = {1995},
 publisher = {Princeton, NJ: Princeton University Press},
 language = {English},
 zbMATH = {826154},
 Zbl = {0841.35001}
}

@article{FontanaMorpurgo,
 author = {Fontana, Luigi and Morpurgo, Carlo},
 title = {Optimal limiting embeddings for {{\(\Delta\)}}-reduced {Sobolev} spaces in {{\(L^1\)}}},
 fjournal = {Annales de l'Institut Henri Poincar{\'e}. Analyse Non Lin{\'e}aire},
 journal = {Ann. Inst. Henri Poincar{\'e}, Anal. Non Lin{\'e}aire},
 volume = {31},
 number = {2},
 pages = {217--230},
 year = {2014},
 language = {English},
 zbMATH = {6347272},
 Zbl = {1316.46035}
}

@article{Flucher,
 author = {Flucher, Martin},
 title = {Extremal functions for the {Trudinger}--{Moser} inequality in 2 dimensions},
 fjournal = {Commentarii Mathematici Helvetici},
 journal = {Comment. Math. Helv.},
 volume = {67},
 number = {3},
 pages = {471--497},
 year = {1992},
 language = {English},
 zbMATH = {86164},
 Zbl = {0763.58008}
}

@article{CassaniRufTarsi,
  author  = {Cassani, Daniele and Ruf, Bernhard and Tarsi, Cristina},
  title   = {Best constants in a borderline case of second-order
             {M}oser type inequalities},
  journal = {Ann. Inst. Henri Poincar{\'e}, Anal. Non Lin{\'e}aire},
  volume  = {27},
  number  = {1},
  pages   = {73--93},
  year    = {2010}
}

@article{MartinazziBMNavier,
 author = {Martinazzi, Luca},
 title = {Classification of solutions to the higher order {Liouville}'s equation on {{\({\mathbb{R}^{2m}}\)}}},
 fjournal = {Mathematische Zeitschrift},
 journal = {Math. Z.},
 volume = {263},
 number = {2},
 pages = {307--329},
 year = {2009},
 language = {English},
 zbMATH = {5603857},
 Zbl = {1172.53026}
}

@article{Lin1998,
  author       = {Lin, Chang-Shou},
  title        = {A classification of solutions of a conformally invariant fourth order equation in {R}$^4$},
  journal      = {Comment. Math. Helv.},
  volume       = {73},
  year         = {1998},
  number       = {2},
  pages        = {206--231}
}

@article{PreviousPaper,
  author = {do {\'O}, Jo{\~a}o Marcos and
            Sani, Federica and Silva, Victor Rafael},
  title = {A {Brezis--Merle} inequality for the biharmonic
           operator with {Dirichlet} boundary conditions
           in four-dimensional balls},
  journal = {arXiv preprint arXiv:2610.27212},
  year = {2026},
  eprint = {2610.27212},
  archivePrefix = {arXiv},
  primaryClass = {math.AP},
  url = {https://arxiv.org/abs/2610.27212}
}

@article{Alberico,
 author = {Alberico, Angela and Ferone, Vincenzo},
 title = {Regularity properties of solutions of elliptic equations in {{\(\mathbb{R}^ 2\)}} in limit cases},
 fjournal = {Atti della Accademia Nazionale dei Lincei. Classe di Scienze Fisiche, Matematiche e Naturali. Serie IX. Rendiconti Lincei. Matematica e Applicazioni},
 journal = {Atti Accad. Naz. Lincei, Cl. Sci. Fis. Mat. Nat., IX. Ser., Rend. Lincei, Mat. Appl.},
 volume = {6},
 number = {4},
 pages = {237--250},
 year = {1995},
 language = {English},
 zbMATH = {915283},
 Zbl = {0860.35015}
}

\end{document}